\documentclass[]{amsart} %%pdf output
\usepackage[usenames]{color}
\usepackage[active]{srcltx}
\usepackage[margin=2cm]{geometry}
\usepackage{multicol}
\usepackage{array}
\usepackage{mathtools}
\usepackage{enumerate}
\usepackage[all,color]{xy} \xyoption{arc}\xyoption{poly}
\usepackage{amssymb}
\usepackage{tikz}
\usepackage{pgf}
\usetikzlibrary{shapes}
\usetikzlibrary{calc}

\newtheorem{lemma}[equation]{Lemma}
\newtheorem{prop}[equation]{Proposition}
\newtheorem{thm}[equation]{Theorem}
\newtheorem{cor}[equation]{Corollary}

\newtheorem{defn}[equation]{Definition}

\newtheorem{question}[equation]{Question}
\theoremstyle{definition}
\newtheorem{exmp}[equation]{Example}
\newtheorem{example}[equation]{Example}

\newtheorem{rmk}[equation]{Remark}

\numberwithin{equation}{section}
\newcommand{\N}{\mathbf{N}}

\newcommand{\R}{\mathbf{R}}

\newcommand{\chs}[2]{\operatorname{chr}^{#1}(#2)}
\newcommand{\chp}{chromatic polynomial}
\newcommand{\chrps}[3]{\chi^{#3}(#1,#2)}

\newcommand{\Mb}{M\"obius}
\newcommand{\euc}{Euler characteristic}

\newcommand{\m}{morphism}

\title{Chromatic polynomials of simplicial complexes}

\author{Jesper M.~M\o ller}
\address{Institut for Matematiske Fag\\
  Universitetsparken 5\\
  DK--2100 K\o benhavn}
\email{moller@math.ku.dk}
\urladdr{http://www.math.ku.dk/~moller}

\author{Gesche Nord}
\address{KdV Instituut voor wiskunde \\
Universeit van Amsterdam}
\email{Gesche.N@gmx.de}
\urladdr{}

\thanks{The first author was supported by the Danish National Research
  Foundation (DNRF) through the Centre for Symmetry and Deformation}
\thanks{We thank Eric Babson whose questions, in November $2011$ at
  the WATACBA workshop in Buenos Aires, led to this note.}
\usepackage[bookmarks=true,bookmarksopen=false]{hyperref}
\hypersetup{
   pdftitle = {},
   pdfauthor = {Jesper Michael Møller},
   pdfpagemode = {UseOutlines},
   pdfstartview = {FitH},
   pdfborder = {0 0 0},
   backref = {true},
   colorlinks = {true},
   urlcolor = {blue},
   citecolor = {blue},
   linkcolor = {blue},
   pdftoolbar = {true},
}


\begin{document}
\date{\today}
\maketitle
%\tableofcontents

\begin{abstract}
  In this note we consider $s$-\chp s of finite simplicial complexes.
  The $s$-\chp s of simplicial complexes are higher dimensional
  analogues of chromatic polynomials for graphs.
\end{abstract}

\section{Introduction}
\label{sec:intro}

Let $K$ be a finite simplicial complex with vertex set $V(K) \neq
\emptyset$ and let $r \geq 1$ and $s \geq 1$ be two natural numbers. A
map $\mathrm{col} \colon V(K) \to \{1,2,\ldots,r\}$ is an {\em
  $(r,s)$-coloring\/} of $K$ if there are no monochrome $s$-simplices
in $K$ \cite{DMN2010}.  We write $\chrps Krs$ for the number of
$(r,s)$-colorings of $K$.

\begin{defn}\label{defn:schpK}
  The {\em $s$-chromatic polynomial\/} of $K$ is the function $\chrps
  Krs$ of $r$.  The {\em $s$-chromatic number\/} of $K$, $\chs sK$, is
  the minimal $r \geq 1$ with $\chrps Krs>0$.
\end{defn}

The theorem below shows that $\chrps Krs$ is indeed polynomial in $r$
for fixed $K$ and $s$.  (By notational convention, $[r]_i = r(r-1)
\cdots (r-i+1)$ is the $i$th falling factorial in $r$.)

\begin{thm}\label{lemma:StK}
  The $s$-\chp\ of $K$ is
  \begin{equation*}
    \chrps Krs = \sum_{i=\chs sK}^{|V(K)|} S(K,i,s)[r]_i
  \end{equation*}
  where $S(K,i,s)$ is the number of partitions of $V(K)$ into
  $i$ blocks containing no $s$-simplex of $K$.
\end{thm}

For $s=1$, an $(r,1)$-coloring of $K$ is a usual graph coloring,
$\chrps Kr1$ is the usual \chp , and $\chs 1K$ the usual chromatic
number of the $1$-skeleton of $K$.  In general, $\chrps Krs$ depends
only on the $s$-skeleton of $K$.  Although the higher $s$-\chp s for
simplicial complexes are analogues of $1$-\chp s for graphs we shall
shortly see that there are structural differences between the cases
$s=$ and $s>1$.

Figure~\ref{fig:MBcol} shows a triangulation $\mathrm{MB}$ of the \Mb\
band. To the left is a $(5,1)$- and to the right a $(2,2)$-coloring of
$\mathrm{MB}$. The \chp s and chromatic numbers \footnote{The
  computations behind the examples of this note were carried out in
  the computer algebra system Magma \cite{magma}.} of $\mathrm{MB}$
are
\begin{equation*}
  \chrps {\mathrm{MB}}rs =
  \begin{cases}
     r^5 - 10 r^4 + 35 r^3 - 50 r^2 + 24 r & s= 1\\
     r^5 - 5 r^3 + 5 r^2 - r & s= 2\\
     r^5 & s \geq 3
  \end{cases} \qquad
  \chs s{\mathrm{MB}} =
  \begin{cases}
    5 & s=1 \\ 2 & s=2 \\ 1 & s \geq 3
  \end{cases}
\end{equation*}
%%%%%%%%%%%%%%%%%%%%
% \begin{equation*}
%  \chrps {\mathrm{MB}}r1 = r^5 - 10 r^4 + 35 r^3 - 50 r^2 + 24 r, \qquad 
%  \chrps {\mathrm{MB}}r2= r^5 - 5 r^3 + 5 r^2 - r, \qquad
%  \chrps {\mathrm{MB}}rs= r^5, \quad s \geq 3,
% \end{equation*}
%and the chromatic numbers are $\chs 1{\mathrm{MB}} = 5$, $\chs
%2{\mathrm{MB}} = 2$, and $\chs s{\mathrm{MB}} = 1$, $s \geq 3$.  

% When $K$ has at least one vertex then $\chrps Kr0=0$ for all $r \geq
% 0$ since all vertices will be monochrome. If the vertex set of $K$ is
% empty, $K= \{\emptyset\}$, then $\chrps {\{\emptyset\}}r0=1$ for all $r
% \geq 0$.

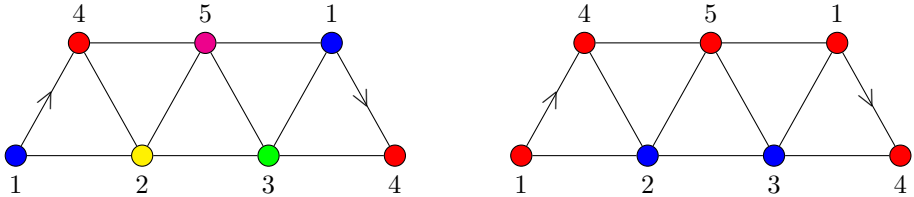
\begin{figure}[h]
  \centering
\begin{tikzpicture}
  [vertex/.style= {shape=circle,  fill=#1!100, minimum size =
  8pt, inner sep =0pt,draw}, vertex/.default=black, xscale=1.68, yscale=1.5]

\coordinate [vertex=red, label=above:$4$] (A1) at (-1,1);
\coordinate [vertex=red, label=above:$5$] (A2) at (0,1);
\coordinate [vertex=red, label=above:$1$] (A3) at (1,1);

\coordinate [vertex=red, label=below:$1$] (B1) at (-1.5,0);
\coordinate [vertex=blue, label=below:$2$] (B2) at (-.5,0);
\coordinate [vertex=blue, label=below:$3$] (B3) at (.5,0);
\coordinate [vertex=red, label=below:$4$] (B4) at (1.5,0);

\draw (B1) -- node[sloped] {$>$} (A1) -- (A2) -- (A3) -- node[sloped]
{$>$} (B4);
\draw (B1) -- (B2) -- (B3) -- (B4);
\draw (A1) -- (B2) -- (A2) -- (B3) -- (A3);

\begin{scope}[xshift=-4cm]
 
\coordinate [vertex=red, label=above:$4$] (A1) at (-1,1);
\coordinate [vertex=magenta, label=above:$5$] (A2) at (0,1);
\coordinate [vertex=blue, label=above:$1$] (A3) at (1,1);

\coordinate [vertex=blue, label=below:$1$] (B1) at (-1.5,0);
\coordinate [vertex=yellow, label=below:$2$] (B2) at (-.5,0);
\coordinate [vertex=green, label=below:$3$] (B3) at (.5,0);
\coordinate [vertex=red, label=below:$4$] (B4) at (1.5,0);

\draw (B1) -- node[sloped] {$>$} (A1) -- (A2) -- (A3) -- node[sloped]
{$>$} (B4);
\draw (B1) -- (B2) -- (B3) -- (B4);
\draw (A1) -- (B2) -- (A2) -- (B3) -- (A3);
\end{scope}

\end{tikzpicture}

\caption{A $(5,1)$-coloring and a $(2,2)$-coloring of a $5$-vertex
  triangulated M\"obius band $\mathrm{MB}$}
  \label{fig:MBcol}
\end{figure}

\subsection{Notation}
\label{sec:notation}
We shall use the following notation throughout the paper:
%\begin{multicols}{2}
\begin{description}
\item[$K$] a finite simplicial complex
\item[$K^s$]  the $s$-skeleton of $K$
\item[$F^s(K)$]  the set of $s$-simplices $K$ 
%\\

\item[$\# V$ or $|V|$] the number of elements in the finite set $V$

\item[$V(K)$] the vertex set $\bigcup K$ of $K$ and $m(K)=|V(K)$ is
  the number of vertices in $K$

\item[{$D[V]$}] the complete simplicial complex of {\em all\/} subsets
  of the finite set $V$
\item[{$[m]$}] the finite set $\{1,\ldots,m\}$ of cardinality $m$
\item[{$[r]_i$}] the $i$the falling factorial polynomial $[r]_i = i!
  \binom ri$ in $r$
\item[$P(a,b)$] the open interval $(a,b)$ in the poset $P$
\end{description}
%\end{multicols}

%\item[{$\partial D[V]$}] The finite simplicial complex of all {\em proper}
%  subsets of $V$ ($|\partial D[V]|=S^{|V|-2}$)
%\item[$m_+$] For any integer $m \geq 0$, $m_+=\{0,1,\ldots,m\}$ denotes a
%set of cardinality $m+1$.
%\item[{$[n]$}] the set with $n$ element for any integer $n \geq 0$
%\item[$\Delta^n$] the full simplex $D[n+1]$ on $n+1$ vertices
%\item[$S^n$] the  simplicial complex of all proper subsets of $[n+1]$

% For instance, we shall see that
%   \begin{equation*}
%     \chrps {D[s+1]}rs = r^{s+1}-r, \qquad
%     \chrps {D[s+2]^s}rs = r^{s+2} - (s+2)r^2 + (s+1)r 
%   \end{equation*}
%   are the $s$-chromatic polynomials of the standard $s$-simplex and
%   the standard $s$-sphere.

\section{Three ways to the $s$-\chp\ of a simplicial complex}
\label{sec:schp}

In this section we present three different to approaches to the
$s$-\chp\ $\chrps Krs$:
\begin{itemize}
\item Theorem~\ref{thm:connindep} via $1$-\chp s of graphs;
\item Theorem~\ref{thm:LsK}
via the \Mb\ function for the $s$-chromatic lattice;
\item Theorem~\ref{lemma:StK} via the simplicial $s$-Stirling numbers
  of the second kind.
\end{itemize}

\subsection{Block-connected $s$-independent vertex partitions}
\label{sec:conpartV}

Let $s \geq 1$ be a natural number.

\begin{defn}\label{defn:independent}
Let $B \subset V(K)$ be a set of vertices of $K$. Then
\begin{itemize}
\item $B$ is {\em $s$-independent\/} if $B$ contains no $s$-simplex of
  $K$;
\item $B$ is  {\em connected\/} if $K \cap D[B]$ is a
  connected simplicial complex;
\item the {\em connected components\/} of $B$ are the maximal
  connected subsets of $B$.
 \end{itemize}
\end{defn}

\begin{defn}\label{defn:cP}
  Let $P$ be a partition of $V(K)$.
  \begin{itemize}
  \item The {\em   graph\/} $G_0(P)$ of $P$ is the simple graph
    whose vertices are the blocks of $P$ and with two blocks connected
    by and edge if their union is connected;
  \item The {\em block-connected refinement\/} $P_0$ of $P$ is the
    refinement whose blocks are the connected components of the blocks
    of $P$;
  \item $P$ is {\em block-connected\/} if the blocks of $P$ are
    connected (ie if $P=P_0$).
  \end{itemize}
\end{defn}

\begin{lemma}\label{lemma:diffblocks}
  Let $P$ be a partition of $V(K)$. If two different blocks of the
  block-connected refinement $P_0$ are connected by an edge in the
    graph $G_0(P_0)$ of $P$ then they lie in different blocks of
  $P$.
\end{lemma}
\begin{proof}
  The connected components of the blocks of $P$ are maximal with
  respect to connectedness.
\end{proof}

% /home/moller/underv/algtop/magma/chrpol.prg  function Cs(K)
\begin{defn}
  $\mathrm{BCP}^s(K)$ is the set of all block-connected $s$-independent
  partitions of $V(K)$.
\end{defn}

Recall that $\chrps {G_0(P)}r1$ is the $1$-\chp\ of the simple graph
$G_0(P)$ of the partition $P$.

%% Cs(K) in chrpol.prg
\begin{thm}\label{thm:connindep}
  The $s$-\chp\ for $K$ is the sum
  \begin{equation*}
    \chrps Krs = \sum_{P \in \mathrm{BCP}^s(K)} \chrps {G_0(P)}r1
  \end{equation*}
  of the $1$-chromatic polynomials and 
 the $s$-chromatic number of $K$ is the minimum  
\begin{equation*}
    \chs sK = \min_{P \in \mathrm{BCP}^s(K)} \chs 1{G_0(P)}
\end{equation*}
of the $1$-chromatic numbers for the   graphs of all the
block-connected $s$-independent partitions of $V(K)$.
\end{thm}
\begin{proof}
  Let $\mathrm{col} \colon V(K) \to [r] $ be an $(r,s)$-coloring of
  $K$.  The monochrome partition $P(\mathrm{col})$ of $V(K)$ is the
  $s$-independent partition whose blocks are the nonempty monochrome
  sets of vertices $\{\mathrm{col}=i\}$ for $i \in [r]$.  The
  block-connected refinement $P(\mathrm{col})_0$ of the monochrome
  partition is a block-connected $s$-independent partition of $K$. The
  original coloring $\mathrm{col}$ of $K$ is also a coloring of the
  graph $G_0(P(\mathrm{col})_0)$ of $P(\mathrm{col})_0$ for, by
  Lemma~\ref{lemma:diffblocks}, distinct vertices of $1$-simplices of
  this graph have distinct colors. We have shown that any
  $(r,s)$-coloring $\mathrm{col}$ of $K$ induces an $(r,1)$-coloring
  $\mathrm{col}_0$ of the graph $G_0(P(\mathrm{col})_0)$ of the
  block-connected refinement of the monochrome partition.

  Let $P \in \mathrm{BCP}^s(K)$ be a block-connected $s$-independent
  partition of $V(K)$ and $\mathrm{col}_0 \colon P \to \{1,\ldots,r\}$
  an $(r,1)$-coloring of its graph $G_0(P)$.  Then $\mathrm{col}_0$
  determines a map $\mathrm{col} \colon V(K) \to [r]$ that is constant
  on the blocks of $P$.  An $s$-simplex of $K$ can not be monochrome
  under $\mathrm{col}$ as it intersects at least two different blocks
  of $P$ connected by an edge of $G_0(P)$. Thus $\mathrm{col}$ is an
  $(r,s)$-coloring of $K$.

  These two constructions are inverses of each other. 
\end{proof}

% \begin{figure}[t]
%   \centering
% \begin{equation*} 
% %  \UseCrayolaColors
%   \xy 0;/r.20pc/: 
%   (0,2)*\frm<20pt>{o};  
%   (0,0)*{\textcolor{red}{\bullet}}="A1"; 
%   (5,2)*{\textcolor{red}{\bullet}}="A2"; 
%   (-5,5)*{\textcolor{red}{\bullet}}="A3";
%   "A1"; "A2" **\dir{-};
%   "A1"; "A3" **\dir{-};
%   "A2"; "A3" **\dir{-};
%   (25,2)*\frm<20pt>{o};  
%   (25,-2)*{\textcolor{blue}{\bullet}}="B1"; 
%   (30,2)*{\textcolor{blue}{\bullet}}="B2"; 
%   (20,5)*{\textcolor{blue}{\bullet}}="B3";
%   (27,7)*{\textcolor{blue}{\bullet}}="B4";
%   "B1"; "B2" **\dir{-};
%   "B1"; "B3" **\dir{-};
%   "B2"; "B3" **\dir{-};
%   "B2"; "B4" **\dir{-};
%   "B1"; "B4" **\dir{-};
%   "B3"; "B4" **\dir{-};
%   "A2"; "B3" **\dir{-};
%   (-25,2)*\frm<20pt>{o};  
%   (-25,0)*{\textcolor{red}{\bullet}}="C1"; 
%   (-20,2)*{\textcolor{red}{\bullet}}="C2"; 
%   (-30,5)*{\textcolor{red}{\bullet}}="C3";
%   (-22,-4)*{\textcolor{red}{\bullet}}="C4";
%   "C1"; "C2" **\dir{-};
%   "C1"; "C3" **\dir{-};
%   "C2"; "C3" **\dir{-};
%   "C2"; "C4" **\dir{-};
%   "C1"; "C4" **\dir{-};
%   (-13,2)*=<4cm,2cm>\frm{.}; 
%   (25,2)*=<1.8cm,2cm>\frm{.};
%   "C4"; "B1" **\dir{-};
%   \endxy 
% \end{equation*}
%     \caption{Colors on the graph $G_0(C)$ for $C \in \mathrm{BCP}^s(K)$}
%   \label{fig:CsK}
% \end{figure}

% \frm<20pt><red>{o};
%\xyoption{color};
%%*[F++:red]

\begin{rmk}[The minimal block-connected $s$-independent partition]
 \label{rmk:C0}
Let $C_0=\{\{v\} \mid v \in V(K)\}$ be the block-connected
$s$-independent partition of $V(K)$ whose blocks are singletons. The
  graph $G_0(C_0) = K^1$ is the $1$-skeleton of $K$.  Thus
the $1$-\chp\ of the $1$-skeleton of $K$ is always one of the
polynomials in the sum of Theorem~\ref{thm:connindep}. If $K$ is
$1$-dimensional, $\mathrm{BCP}^1(K)$ consists only of the partition
$C_0$ and Theorem~\ref{thm:connindep} simply says that the $1$-\chp\
of a simplicial complex is the $1$-\chp\ of its $1$-skeleton.
\end{rmk}
  
\begin{exmp}[The block-connected $2$-independent partitions for {$D[3]$}]
% /home/moller/underv/algtop/magma/chrpol.prg 
% Cs(K,s) computes \mathrm{BCP}^s(K) and related things
  The $2$-simplex $D[3]$ has $4$ block-connected $2$-independent
  partitions $C_0$, $\{\{1\},\{2,3\}\}$, $\{\{2\},\{1,3\}\}$, and
  $\{\{3\},\{1,3\}\}$. The   graph of $C_0$ is the
  complete graph $K_3$, the $1$-skeleton of $D[3]$. The graphs of the
  other three partitions are all the complete graph $K_2$. Thus the
  $2$-\chp\ of $D[3]$ is $\chrps {D[3]}r2 = \chrps {K_3}r1 + 3\chrps
  {K_2}r1 = [r]_3 + 3[r]_2 = [r]_2(r+1) = r^3-r$ and the $2$-chromatic
  number is $\chs 2{D[3]} = 2$.
%%%%%
%  Similarly, the $2$-\chp\ of the $3$-simplex $D[4]$ is
%  $\chi(D[4],2) = [r]_4 + 6[r]_3 + 3[r]_2 = r^4 - 4r^2 + 3r$.
\end{exmp}

\begin{exmp}[A $(2,2)$-coloring and the graph of the block-connected
  refinement of its monochrome partition]\label{rmk:T2}
  The picture below illustrates a $(2,2)$-coloring of a $9$-vertex
  triangulation of the torus
  
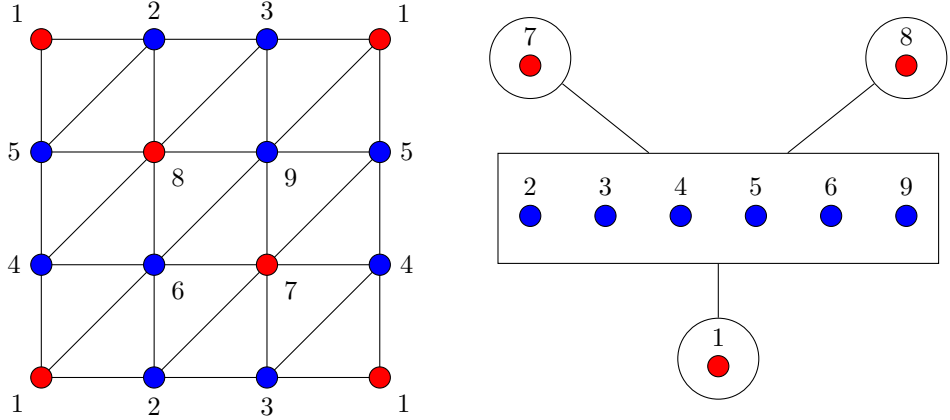
\begin{center}
 \begin{tikzpicture}
  [vertex/.style= {shape=circle,  fill=#1!100, minimum size =
  8pt, inner sep =0pt,draw}, vertex/.default=black]
  
\begin{scope}[scale=1.5]
  \draw [step = 1cm] (1,1) grid (4,4);
  \draw (1,1) -- (4,4);
  \draw (3,1) -- (4,2);
  \draw (2,1) -- (4,3);
  \draw (1,2) -- (3,4);
  \draw (1,3) -- (2,4);

  \node (a) at (1,4) [vertex=red,label=above left:$1$] {};
  \node (b) at (2,4) [vertex=blue,label=above:$2$] {};
  \node (c) at (3,4) [vertex=blue,label=above:$3$] {};
  \node (d) at (4,4) [vertex=red,label=above right:$1$] {};

  \node  at (1,3) [vertex=blue,label=left:$5$] {};
  \node  at (2,3) [vertex=red,label=below right:$8$] {};
  \node  at (3,3) [vertex=blue,label=below right:$9$] {};
  \node  at (4,3) [vertex=blue,label=right:$5$] {};

  \node  at (1,2) [vertex=blue,label=left:$4$] {};
  \node  at (2,2) [vertex=blue,label=below right:$6$] {};
  \node  at (3,2) [vertex=red,label=below right:$7$] {};
  \node  at (4,2) [vertex=blue,label=right:$4$] {};

  \node  at (1,1) [vertex=red,label=below left:$1$] {};
  \node  at (2,1) [vertex=blue,label=below:$2$] {};
  \node  at (3,1) [vertex=blue,label=below:$3$] {};
  \node  at (4,1) [vertex=red,label=below right:$1$] {};
\end{scope}
  \begin{scope}[xshift=7cm,yshift=1.25cm]
    \node (C1) at (3.5,2.5) 
    [rectangle,draw]  {\phantom{\rule[0pt]{160pt}{35pt}}}; 
     \node (C2) at (3.5,0.5) 
    [circle,draw]  {\phantom{\rule[0pt]{15pt}{15pt}}};
    \node (C3) at (1,4.5) 
    [circle,draw]  {\phantom{\rule[0pt]{15pt}{15pt}}};
     \node (C4) at (6,4.5) 
    [circle,draw]  {\phantom{\rule[0pt]{15pt}{15pt}}};

    \node at (1,2.4) [vertex=blue, label= above:$2$] {};
    \node at (2,2.4) [vertex=blue, label= above:$3$] {};
    \node at (3,2.4) [vertex=blue, label= above:$4$] {};
    \node at (4,2.4) [vertex=blue, label= above:$5$] {};
    \node at (5,2.4) [vertex=blue, label= above:$6$] {};
    \node at (6,2.4) [vertex=blue, label= above:$9$] {};

    \node at (3.5,0.4) [vertex=red, label= above:$1$] {};
    \node at (1,4.4) [vertex=red, label= above:$7$] {};
    \node at (6,4.4) [vertex=red, label= above:$8$] {};

    \draw (C1) -- (C2);
    \draw (C1) -- (C3);
    \draw (C1) -- (C4);
  \end{scope}
\end{tikzpicture}
\end{center}
and its corresponding graph. There are $6937$ block-connected
partitions of the vertex set, and $3$ of them has the graph shown
above. The $2$-chromatic polynomial is $21[r]_2+ 742[r]_3 + 3747[r]_4
+ 4908[r]_5 + 2295[r]_6 + 444[r]_7 + 36[r]_8 + [r]_9 = [r]_2(r^7 + r^6
- 17r^5 + 10r^4 + 82r^3 - 116r^2 - 23r + 67)$ and the $2$-chromatic
number is $2$.
\end{exmp}

\begin{exmp}[The $(r,2)$-colorings of a simplicial complex $K$]\label{exmp:K}
  Let $K$ be the pure $2$-dimensional complex with facets $F^2(K) =
  \{\{1,2,3\},\{2,3,4\},\{4,5,6\}\}$.
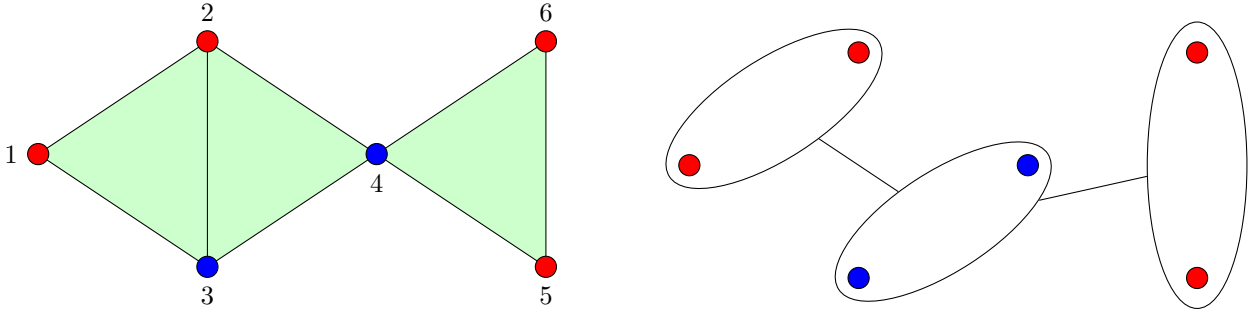
  \begin{center}
    \begin{tikzpicture}
[vertex/.style= {shape=circle,  fill=#1!100, minimum size =
  8pt, inner sep =0pt,draw}, vertex/.default=black,scale=0.75,xscale=1.5]

  \path [fill = green!20] (0,0) -- (2,2) -- (2,-2) -- (0,0); 
  \path [fill = green!20] (4,0) -- (2,2) -- (2,-2) -- (4,0); 
  \path [fill = green!20] (4,0) -- (6,2) -- (6,-2) -- (4,0);

  \draw (2,2) -- (2,-2);
  \draw (0,0) -- (2,2) -- (4,0) -- (6,2) -- (6,-2) -- (4,0) -- (2,-2)
  -- (0,0) -- cycle;

  \draw (0,0) node[vertex=red, label=left:$1$] {};
  \draw (2,2) node[vertex=red, label=above:$2$] {};
  \draw (2,-2) node[vertex=blue, label=below:$3$] {};
  \draw (4,0) node[vertex=blue, label=below:$4$] {};
  \draw (6,-2) node[vertex=red, label=below:$5$] {};
  \draw (6,2) node[vertex=red, label=above:$6$] {};
\end{tikzpicture} \hspace{1cm}
 \begin{tikzpicture}
[vertex/.style= {shape=circle,  fill=#1!100, minimum size =
  8pt, inner sep =0pt,draw},
vertex/.default=black,scale=0.75]

 \draw (0,0) node[vertex=red] {};
  \draw (3,2) node[vertex=red] {};
  \draw (3,-2) node[vertex=blue] {};
  \draw (6,0) node[vertex=blue] {};
  \draw (9,-2) node[vertex=red] {};
  \draw (9,2) node[vertex=red] {};

  \path (1.5,1) node(a) [ellipse, rotate=-57, draw] 
 {\phantom{\rule[0pt]{20pt}{60pt}}};

 \path (4.5,-1) node(b) [ellipse, rotate=-57, draw] 
 {\phantom{\rule[0pt]{20pt}{60pt}}};

 \path (9,0) node(c) [ellipse, rotate=0, draw] 
 {\phantom{\rule[0pt]{20pt}{70pt}}};

 \draw (a) -- (b) -- (c);

\end{tikzpicture} \hspace{2cm}
\end{center}
The picture shows a $(2,2)$-coloring of $K$ and the corresponding
$(2,1)$-coloring of the associated graph, $G_0(P_0)$, the block
connected refinement of the monochrome partition
$P=\{\{1,2,5,6\},\{3,4\}\}$.  Table~\ref{tab:G0PforK} shows the graphs
$G_0(P)$ for all block connected partitions $P \in \mathrm{BCP}^2(K)$.
For each graph, the table records its $1$-chromatic polynomial and its
$1$-chromatic number. The $2$-chromatic polynomial of $K$ is $\chrps
K22 = 15[r]_2 + 73[r]_3 + 62[r]_4 + 15[r]_5 + [r]_6 =
[r]_2(r-1)(r+1)(r^2+r-1)$ and the $2$-chromatic number is $\chs 2K =
2$.

\begin{table}[t]
  \centering
\begin{tabular}[t]{c|>{\centering\arraybackslash}m{2.2cm}|c|c}
    $\#$  in $\mathrm{BCP}^2(K)$ & $G_0(P)$ & $\chrps {G_0(P)}r1$ 
    & $\chs 1{G_0(P)}$ \\ \hline
    $1$ & \vspace{2mm}
    \begin{tikzpicture}
      [vertex/.style= {shape=circle,  fill=#1!100, minimum size =
  3pt, inner sep =0pt,draw}, vertex/.default=black,scale=.25,xscale=1.3]
    \draw (0,0) node[vertex] {} -- 
    (2,2) node[vertex] {} -- 
   (4,0) node[vertex] {} -- 
   (6,2) node[vertex] {} -- 
   (6,-2) node[vertex] {} -- 
   (4,0) node[vertex] {} --
    (2,-2) node[vertex] {} -- 
    (0,0) -- cycle;
    \draw (2,2) -- (2,-2);
    \end{tikzpicture} &
    $r(r-1)^2(r-2)^3$ &  $3$  \\   
    %% $2[r]_3 + 10[r]_4 + 7[r]_5 + [r]_6$ \\
    $1$ &
     \begin{tikzpicture}
      [vertex/.style= {shape=circle,  fill=#1!100, minimum size =
  3pt, inner sep =0pt,draw}, vertex/.default=black,scale=.25,xscale=1.3]
    \draw (0,2) node[vertex] {} -- 
    (0,-2) node[vertex] {} -- 
   (2,0) node[vertex] {} -- 
   (4,0) node[vertex] {} -- 
   (6,0) node[vertex] {};
    \draw (0,2) -- (2,0);
    \end{tikzpicture} &
    $r(r-1)^3(r-2)$ & $3$ \\        %% $4[r]_3+5[r]_4+[r]_5$ \\
    $3$ &
      \begin{tikzpicture}
      [vertex/.style= {shape=circle,  fill=#1!100, minimum size =
  3pt, inner sep =0pt,draw}, vertex/.default=black,scale=.25,xscale=1.3]
    \draw (0,0) node[vertex] {} -- 
    (2,-2) node[vertex] {} -- 
   (4,0) node[vertex] {} -- 
   (6,0) node[vertex] {};
   \draw (0,0) --  
   (2,2) node[vertex] {} -- (4,0); 
   \draw (2,2) -- (2,-2);
    \end{tikzpicture} &
    $r(r-1)^2(r-2)^2$ & $3$ \\  %% $2[r]_3 + 4[r]_4 + [r]_5$ \\
    $4$ &
      \begin{tikzpicture}
      [vertex/.style= {shape=circle,  fill=#1!100, minimum size =
  3pt, inner sep =0pt,draw}, vertex/.default=black,scale=.25,xscale=1.3]
    \draw (0,2) node[vertex] {} -- 
    (0,-2) node[vertex] {} -- 
   (2,0) node[vertex] {} -- 
   (4,2) node[vertex] {} -- 
   (4,-2) node[vertex] {} -- (2,0);
   \draw (0,2) -- (2,0);
    \end{tikzpicture} &
    $r(r-1)^2(r-2)^2$ & $3$ \\    %% $2[r]_3 + 4[r]_4 + [r]_5$ \\
    $16$ &
     \begin{tikzpicture}
      [vertex/.style= {shape=circle,  fill=#1!100, minimum size =
  3pt, inner sep =0pt,draw}, vertex/.default=black,scale=.25,xscale=1.3]
    \draw (0,2) node[vertex] {} -- 
    (0,-2) node[vertex] {} -- 
   (2,0) node[vertex] {} -- 
   (4,0) node[vertex] {}; 
   \draw (0,2) -- (2,0);
    \end{tikzpicture} &
    $r(r-1)^2(r-2)$  & $3$ \\  %% $2[r]_3+[r]_4$ \\
    $3$ &
      \begin{tikzpicture}
      [vertex/.style= {shape=circle,  fill=#1!100, minimum size =
  3pt, inner sep =0pt,draw}, vertex/.default=black,scale=.25,xscale=1.3]
    \draw (0,0) node[vertex] {} -- 
    (2,0) node[vertex] {} -- 
   (4,0) node[vertex] {} -- 
   (6,0) node[vertex] {}; 
    \end{tikzpicture} &
   $r(r-1)^3$ &  $2$ \\        %%$[r]_2 + 3[r]_3 + [r]_4$ \\
     $12$ &
      \begin{tikzpicture}
      [vertex/.style= {shape=circle,  fill=#1!100, minimum size =
  3pt, inner sep =0pt,draw}, vertex/.default=black,scale=.25,xscale=1.3]
    \draw (0,0) node[vertex] {} -- 
    (2,0) node[vertex] {} -- 
   (4,0) node[vertex] {}; 
    \end{tikzpicture} &
    $r(r-1)^2$ & $2$ 
  %   \\ \hline             %% $[r]_2 + [r]_3$ \\ \hline
%     $\# \mathrm{BCP}^2(K) = 40$ & {} & 
%     $\chrps K22 = 15[r]_2 + 73[r]_3 + 62[r]_4 + 15[r]_5 + [r]_6$
  \end{tabular}
  \caption{The graphs for the block-connected partitions in 
           $\mathrm{BCP}^2(K)$}
  \label{tab:G0PforK}
\end{table}
\end{exmp}  

  \begin{exmp}[The $(r,2)$-colorings of the \Mb\ band]
    The set $\mathrm{BCP}^2(\mathrm{MB})$ of block-connected
    $2$-independent partitions of the triangulated \Mb\ band
    $\mathrm{MB}$ (Figure~\ref{fig:MBcol}) has $36$ elements. There
    are $5,5,15,10,1$ partitions in $\mathrm{BCP}^2(\mathrm{MB})$
    realizing the partitions $[3,2]$, $[3,1,1]$, $[2,2,1]$,
    $[2,1,1,1]$, $[1,1,1,1,1]$ of the integer $|V(\mathrm{MB})|=5$.
    All associated graphs are complete graphs.
   % \marginpar{\textcolor{red}{Because they are manifolds? NO! $T2$}}
    This yields the $2$-\chp\ $\chrps {\mathrm{MB}}r2 = 5[r]_2 +
    20[r]_3 + 10[r]_4 + [r]_5 = [r]_2(r^3 + r^2 - 4r + 1) = r^5 - 5r^3
    + 5r^2 - r$ and the $2$-chromatic number is $\chs 2{\mathrm{MB}} =
    2$.
%%%%%%%%%
%  There are usually many block-connected $s$-independent partitions
%  of the vertex set. For instance  $|\mathrm{BCP}^2(\mathrm{MT})|=463$
%  and  $|\mathrm{BCP}^2(\mathrm{P}2)|=116$ for the simplicial complexes
%   $\mathrm{MT}$ and $\mathrm{P}^2$ from Example~\ref{exmp:MT}.
\end{exmp}

\begin{rmk}[The $\mathcal{S}$-chromatic polynomial of
  $K$]\label{rmk:Schrpol} 
  Let $\mathcal{S}$ be a set of connected subcomplexes of $K$. A set
  $B \subset V(K)$ of vertices is $\mathcal{S}$-independent if $B$ is
  not a superset of any member of $\mathcal{S}$. Let
  $\mathrm{BCP}^\mathcal{S}(K)$ be the set of
  $\mathcal{S}$-independent partitions of $V(K)$. An
  $(r,\mathcal{S})$-coloring is a map $V(K) \to \{1,\ldots,r\}$ such
  that $\#\mathrm{col}(S)>1$ for all $S \in \mathcal{S}$. The number
  of $(r,\mathcal{S})$-colorings of $K$ is
 \begin{equation*}
    \chrps Kr{\mathcal{S}} = 
    \sum_{P \in \mathrm{BCP}^\mathcal{S}(K)} \chrps {G_0(P)}r1
  \end{equation*}
  as one sees by an obvious generalization of
  Theorem~\ref{thm:connindep}. An $(r,s)$-coloring of $K$ is an
  $(r,\mathcal{S})$-coloring of $K$ where $\mathcal{S} = F^s(K)$ is
  the set of $s$-simplices.
\end{rmk}

\subsection{The $s$-chromatic linear program}
\label{sec:linprg}
% /home/moller/underv/algtop/magma/lpcol.prg
Read \cite[\S10]{read68} explains how to construct a linear program
%$\mathrm{LP}^s(K)$ 
with minimal value equal to the $s$-chromatic number $\chs sK$ of $K$.

% /home/moller/students/gesche/read68.pdf

\marginpar{\textcolor{red}{maximal?}}

\begin{defn}
  $M^s(K)$ is the set of all maximal $s$-independent subsets of
  $V(K)$.
\end{defn}

Let $A$ be the $(m(K) \times |M^s(K)|)$-matrix
\begin{equation*}
  A(v,M) =
  \begin{cases}
    1 & v \in M \\
    0 & v \not\in M
  \end{cases}
\end{equation*}
recording which vertices $v \in V(K)$ belong to which maximal $s$-independent
sets $M \in M^s(K)$. Now the $s$-chromatic number
\begin{equation*}
  \chs sK = \min \{ \sum_{M \in M^s(K)} x(M) \mid
  x \colon M^s(K) \to \{0,1\}, \forall v \in V(K) \colon \sum_{M \in M^s(K)}
  A(v,M)x(M) \geq 1 \}
\end{equation*}
is the minimal value of the objective function $\sum_{M \in M^s(K)} x(M)$
in  $|M^s(K)|$
variables $x \colon M^s(K) \to \{0,1\}$, taking values $0$ or $1$, and
$m(K)$ constraints $\sum_{M \in M^s(K)} A(v,M)x(M) \geq 1$, $v \in V(K)$.
%% lpcol.prg

\subsection{The $s$-chromatic lattice}
\label{sec:chrlattice}

Our approach here simply follows Rota's classical method for computing
\chp s from \Mb\ functions of lattices \cite[\S9]{rota64}.  We need
some terminology in order to characterize the monochrome loci for
colorings of $K$. Recall that $F^s(K)$ is the set of $s$-simplices of
$K$.

\begin{defn}
  Let $S \subset F^s(K)$ be a set of $s$-simplices of $K$.
  \begin{itemize}
  \item The equivalence relation $\sim$  is the smallest equivalence
    relation in $S$ 
     such that $s_1 \cap
    s_2 \neq \emptyset \Longrightarrow s_1 \sim s_2$ for all $s_1, s_2
    \in S$;
  \item the connected components of $S$ are the equivalence classes
    under $\sim$;
  \item $\pi_0(S)$ is the set of connected components of $S$;
  \item $S$ is {\em connected\/} if it has at most one component;
  \item $V(S) = \bigcup S$ is the {\em vertex set\/} of $S$
  \item $\pi(S)$ is the partition of $V(K)$ whose blocks are the
    vertex sets of the connected components of $S$ together with the
    singleton blocks $\{v\}$, $v \in V(K)-V(S)$, of vertices not in
    any simplex in $S$;
  \item $S$ is {\em closed\/} if $S$ contains any $s$-simplex in $K$
    contained in the vertex set of $S$, ie if
\begin{equation*}
  \{ \sigma \in F^s(K) \mid \sigma \subset V(S) \} = S
\end{equation*}
\item the {\em closure\/} of $S$ is the smallest closed set of $s$-simplices
containing $S$.
  \end{itemize}
\end{defn}

For instance, the empty set $S=\emptyset$ of $0$ $s$-simplices is
connected with $0$ connected components. If $K=D[4]$, the set
$\{\{1,2\},\{2,4\}\}$ of $1$-simplices is connected while
$\{\{1,2\},\{3,4\}\}$ has the two components $\{\{1,2\}\}$ and
$\{\{3,4\}\}$.

A set of $s$-simplices is closed if and only if it equals its closure.
For instance in $F^2(D[5])$, the set $\{\{1,2,3\},\{3,4,5\}\}$ is not
closed because its closure is the set of all $2$-simplices in $D[5]$.
The empty set of $s$-simplices, any set of just one $s$-simplex, and
any set of disjoint $s$-simplices are closed.

In this picture the green set of $2$-simplices is
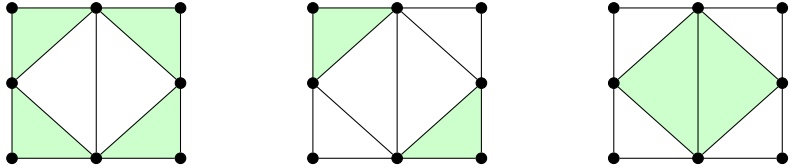
\begin{center}
 \begin{tikzpicture}
  [vertex/.style= {shape=circle,  fill=#1!100, minimum size =
  4pt, inner sep =0pt,draw}, vertex/.default=black]

\begin{scope}[yscale=1]
 \path [fill=green!20] (0,0) -- (0,1) -- (1.12,0) --(0,0); 
 \path [fill=green!20] (0,1) -- (0,2) -- (1.12,2) --(0,1);
 \path [fill=green!20] (1.12,2) -- (2.24,2) -- (2.24,1) --(1.12,2);
 \path [fill=green!20] (2.24,1) -- (2.24,0) -- (1.12,0) --(2.24,1);

  \node (a) at (0,0) [vertex] {};
  \node (b) at (0,1) [vertex] {};
  \node (c) at (0,2) [vertex] {};
  \node (d) at (1.12,2) [vertex] {};
  \node (e) at (2.24,2) [vertex] {};
  \node (f) at (2.24,1) [vertex] {};
  \node (g) at (2.24,0) [vertex] {};
  \node (h) at (1.12,0) [vertex] {};

  \draw (a) -- (c) -- (e) -- (g) -- (a) -- cycle;
  \draw (b) -- (d) -- (f);
  \draw (b) -- (h) -- (f);
  \draw (d) -- (h);
\end{scope}

\begin{scope}[yscale=1,xshift=4cm]
 
 \path [fill=green!20] (0,1) -- (0,2) -- (1.12,2) --(0,1);
 
 \path [fill=green!20] (2.24,1) -- (2.24,0) -- (1.12,0) --(2.24,1);

  \node (a) at (0,0) [vertex] {};
  \node (b) at (0,1) [vertex] {};
  \node (c) at (0,2) [vertex] {};
  \node (d) at (1.12,2) [vertex] {};
  \node (e) at (2.24,2) [vertex] {};
  \node (f) at (2.24,1) [vertex] {};
  \node (g) at (2.24,0) [vertex] {};
  \node (h) at (1.12,0) [vertex] {};

  \draw (a) -- (c) -- (e) -- (g) -- (a) -- cycle;
  \draw (b) -- (d) -- (f);
  \draw (b) -- (h) -- (f);
  \draw (d) -- (h);
\end{scope}

\begin{scope}[yscale=1,xshift=8cm]

 \path [fill=green!20] (0,1) -- (1.12,2) -- (1.12,0) --(0,1);
 \path [fill=green!20] (1.12,2) -- (1.12,0) -- (2.24,1) --(1.12,2);

  \node (a) at (0,0) [vertex] {};
  \node (b) at (0,1) [vertex] {};
  \node (c) at (0,2) [vertex] {};
  \node (d) at (1.12,2) [vertex] {};
  \node (e) at (2.24,2) [vertex] {};
  \node (f) at (2.24,1) [vertex] {};
  \node (g) at (2.24,0) [vertex] {};
  \node (h) at (1.12,0) [vertex] {};

  \draw (a) -- (c) -- (e) -- (g) -- (a) -- cycle;
  \draw (b) -- (d) -- (f);
  \draw (b) -- (h) -- (f);
  \draw (d) -- (h);
\end{scope}
\end{tikzpicture}
\end{center}
connected and not closed, closed and not connected, closed and
connected, respectively.

The partition $\pi(S)$ has $|\pi(S)| =
|\pi_0(S)|+m(K)-|V(S)|$ blocks.

\begin{lemma}\label{lemma:closedcomp}
  Let $S$ be a set of $s$-simplices in $K$ and $S_0$ a connected
  component of $S$. Then $S_0$ is closed if and only if
  \begin{equation*}
    \{ \sigma \in F^s(K) \mid \sigma \subset V(S_0) \} \subset S
  \end{equation*}
\end{lemma}
\begin{proof}
  Since the condition is certainly necessary we only need to see that
  it is sufficient. Let $\sigma$ be an $s$-simplex in $K$ with all its
  vertices in $V(S_0)$. Then $\sigma$ lies in $S$ by assumption. But
  $\sigma$ is equivalent to all elements of the equivalence class
  $S_0$. Thus $\sigma \in S_0$.
\end{proof}

\begin{lemma}\label{lemma:meet}
  Let $S$ and $T$ be sets of $s$-simplices in $K$.
  \begin{enumerate}
  \item \label{lemma:meet1} If $S$ and $T$ are closed, so is $S \cap T$.
  \item \label{lemma:meet2} If $S$ and $T$ have closed connected
    components, so does $S \cap T$
  \end{enumerate}
\end{lemma}
\begin{proof}

  \noindent \eqref{lemma:meet1}
  Let $\sigma$ be an $s$-simplex of $K$ and suppose that $\sigma
  \subset V(S \cap T)$. Then $\sigma \subset V(S)$ an $\sigma \subset
  V(T)$ so that $\sigma \in S$ and $\sigma \in T$ as $S$ and $T$ are
  closed.  

  \noindent \eqref{lemma:meet2} Let $R$ be a connected component of $S
  \cap T$. Let $S_0$ be the connected component of $S$ containing $R$
  and $T_0$ be the connected component of $T$ containing $R$.  Then $R
  \subset S_0 \cap T_0$.  Suppose that $\sigma \in F^s(K)$ is an
  $s$-simplex with $\sigma \subset V(R)$. Then $\sigma \subset V(S_0
  \cap T_0)$ so $\sigma \in S_0 \cap T_0$ by \eqref{lemma:meet1} as
  the connected components $S_0$ and $T_0$ are assumed to be closed.
  In particular, $\sigma \in S \cap T$. According to
  Lemma~\ref{lemma:closedcomp}, the connected component $R$ is closed.
\end{proof}

\begin{defn}\label{defn:LsK}
  The $s$-chromatic lattice of $K$ is the set $L^s(K)$ of all subsets
  of $F^s(K)$ with closed connected components.  $L^s(K)$ is a
  partially ordered by set inclusion.
\end{defn}

The set $L^s(K)$ contains the empty set $\emptyset$ of $s$-simplices
and the set $F^s(K)$ of all $s$-simplices. These two elements of
$L^s(K)$ are distinct when $K$ has dimension at least $s$.

\begin{cor}
  $L^s(K)$ is a finite lattice with $\widehat 0 = \emptyset$,
  $\widehat 1 = F^s(K)$, and meet $S \wedge T = S \cap T$. 
\end{cor}
\begin{proof}
  If $S,T \in L^s(K)$ then $S \cap T$ is also in $L^s(K)$ by
  Lemma~\ref{lemma:meet} and this is clearly the greatest lower
  bound of $S$ and $T$. It is now a standard result that $L^s(K)$ is a
  finite lattice \cite[Proposition 3.3.1]{stanley97}. The join $S \vee
  T$ of $S,T \in L^s(K)$ is the intersection of all supersets $U \in
  L^s(K)$ of $S \cup T$.
\end{proof}
% chrpol.prg contains tools to analyze this lattice

% \begin{defn}\cite[p $105$]{stanley97}\label{defn:geomlattice}
%   A {\em finite geometric lattice\/} is a finite lattice $L$ so that
%   \begin{itemize}
%   \item all maximal chains in $L$ have the same length
%   \item the rank function $\rho$ of $L$ satisfies $\rho(x) + \rho(y)
%     \geq \rho(x \wedge y) + \rho(x \vee y)$
%   \item any element of $L$ is a join of atoms
%   \end{itemize}
%   A {\em finite semi-modular} lattice is a finite lattice $L$
%   satisfying the two first conditions.
% \end{defn}

\begin{exmp}[The $s$-chromatic lattice {$L^s(D[m])$}]
  \label{exmp:L2Dn+} The closed and connected elements of the
  $s$-chromatic lattice $L^s(D[m])$ of the complete simplex $D[m]$ on
  $m>s$ vertices are $\emptyset$ and the $\binom mk$ sets $F^s(D[k])$
  of all $s$-simplices in the subcomplexes $D[k]$ for $s < k \leq m$.
  The map $S \to \pi(S)$ is an iso\m\ between the lattice $L^s(D[m])$
  and the lattice, ordered by refinement, of all partitions of the set
  $[m]$ into blocks of size $>s$ or $1$.  The least element, $\widehat
  0 = (1) \cdots (m)$, is the partition with $m$ blocks and the
  greatest element, $\widehat 1 = (1\cdots m)$, the partition with $1$
  block.  $L^s(D[m])$ is not a graded lattice \cite[p 99]{stanley97}
  in general when $s \geq 2$. To see this, observe that the
  $2$-chromatic lattices $L^2(D[3])$, $L^2(D[4])$, and $L^2(D[4])$ are
  graded but the lattice $L^2(D[6])$ is not graded as it contains two
  maximal chains
  \begin{gather*}
    \widehat 0 = (1)(2)(3)(4)(5)(6) < (123)(4)(5)(6) < 
    (1234)(5)(6) < (12345)(6) < (123456) = \widehat 1 \\
     \widehat 0 = (1)(2)(3)(4)(5)(6) < (123)(4)(5)(6) < (123)(456) < 
    (123456) = \widehat 1
  \end{gather*} 
  of unequal length. In contrast, the $1$-chromatic lattice of any
  finite simplicial complex is always graded and even geometric
  \cite[\S9, Lemma~1]{rota64}. 
\end{exmp}

% Wtmu in chrpol.prg

%   \begin{table}[t]
%     \centering
%     \begin{tabular}{c|ccccccc}
%       $w=1^m$ & $3$ & $4$ & $5$ & $6$ & $7$ & $8$ & $9$\\ \hline
%       $\mu^2_m(w)(\widehat 0,\widehat 1)$ & 
%       $-1$ & $3$ & $-6$ & $0$ & $90$ &
%       $-630$ & $2520$ \\
%       $|L^2_m(w)|$ & 
%       $2$ & $6$ & $17$ & $53$ & $205$ & $871$ & $3162$ 
%     \end{tabular} \qquad
%     \begin{tabular}{c|ccccccccc}
%       $w=3^11^{m-1}$ & $3$ & $4$ & $5$ & $6$ & $7$ & $8$ & $9$ \\ \hline
%       $\mu^2_m(w)(\widehat 0,\widehat 1)$ &
%       $1$ & $0$ &  $-6$ & $30$ & $-90$ & $0$ & $2520$\\
%        $|L^2_m(w)|$ & 
%        $4$ & $9$ & $25$ & $83$ & $307$ & $1242$ & $5516$
%     \end{tabular}
%     \caption{\Mb\ values $\mu^2_m(1^m)(\widehat 0, \widehat 1)$ and
%       $\mu^2_m(3^11^{m-1})(\widehat 0, \widehat 1)$}
%     \label{tab:mum2}
%   \end{table}

  \begin{rmk}[The \Mb\ function for the $s$-chromatic lattices
    {$L^s(D[m])$}] \label{rmk:wlattice} Our discussion of the \Mb\
    function for the lattice $L^s(D[m])$ echoes the exposition of the
    \Mb\ function for the geometric lattice $L^1(D[m])$ of {\em all\/}
    partitions from \cite[Example 3.10.4]{stanley97}.

  Let $w \colon [m] \to \N$ be a function that to every element of
  $[m]$ associates a natural number, thought of as a weight function.
  We write $w=1^{i_1}2^{i_2}\cdots r^{i_r}$, or something similar, for
  the weight function $w$ defined on the set $[m]$ of cardinality $m =
  \sum_j i_j$ and mapping $i_j$ elements to $j$ for $1 \leq
  j \leq r$. The map $w$ extends to a map, also called $w$, defined on
  the set of all nonempty subsets $X$ of $[m]$ given by $w(X) =
  \sum_{x \in X}w(x)$.  Let $L^s_m(w)$ be the lattice of all
  partitions of the set $[m]$ into blocks $X$ that are singletons or
  have weight $w(X)>s$. The non-singleton blocks of the meet $\sigma
  \wedge \tau$ of two partitions $\sigma, \tau \in L^s_m(w)$ are the
  subsets of weight $>s$ of the form $S \cap T$ where $S$ is a block
  in $\sigma$ and $T$ a block in $\tau$.  Write $\mu_m^s(w)$ for the
  \Mb\ function of $L^s_m(w)$.

  In particular, $L^s_m(1^m)$ is a synonym for $L^s(D[m])$ and we are
  primarily interested in the \Mb\ function $\mu_m^s(1^m)$ of the
  uniform weight $w=1^m$. However, the computation of this \Mb\
  function will involve the \Mb\ functions of other weights as
  well. We shall therefore discuss the \Mb\ functions  $\mu_m^s(w)$
  for general weight functions $w$.

  Suppose that $\sigma \in L^s_m(w)$, $\sigma < \widehat 1$, is a
  partition of $[m]$ into singleton blocks or blocks of weight $>s$.
  Let $w(\sigma)$ be the restriction of $w$ to the set of blocks of
  $\sigma$. Thus $w(\sigma)(X) = \sum_{x \in X}w(x)$ for any block $X$
  of $\sigma$. Then the interval
  \begin{equation*}
   L^s_m(w) \supset [\sigma, \widehat 1] = L^s_{|\sigma|} (w(\sigma))
  \end{equation*}
  so that $\mu^s_m(w)(\sigma,\widehat 1) =
  \mu^s_{|\sigma|}(w(\sigma))(\widehat 0, \widehat 1)$.  More
  generally, suppose that $\sigma < \tau$ for some $\tau \in
  L^s_m(w)$. Assume that the partition $\tau$ has blocks $\tau_j$.
  Let $\sigma_j$ be the set of those blocks of $\sigma$ 
  that intersect the block $\tau_j$ of $\tau$. Let $w(\sigma_j)$ be the
  restriction of $w(\sigma)$ to $\sigma_j$. Then the interval
  \begin{equation*}
    L^s_m(w) \supset [\sigma,\tau] =  \prod_j L^s_{|\sigma_j|}(w(\sigma_j)) 
  \end{equation*}
  and therefore the value of the \Mb\ function on the pair
  $(\sigma,\tau)$ 
  \begin{equation*}
    \mu^s_m(w)(\sigma,\tau) = \prod_j 
    \mu^s_{|\sigma_j|}(w(\sigma_j))(\widehat 0, \widehat 1) 
  \end{equation*}
  by the product theorem for \Mb\ functions \cite[Proposition
  3.8.2]{stanley97}. We conclude that the complete \Mb\ functions on
  all the lattices $L^s_m(w)$, are actually determined by the values
  $\mu^s_m(w)(\widehat 0, \widehat 1)$ of these \Mb\ functions on just
  $(\widehat 0, \widehat 1)$. See Equation \eqref{eq:umsw01} for more
  information about these \euc s.

  For the following it is convenient to name the elements of the
  domain $[m]$ of $w$ so that the element $m$ carries minimal weight.
  Assume that $a_m = (1 \cdots m-1)(m)$ is an element of $L^s_m(w)$,
  ie that $w(1)+\cdots+w(m-1)>s$. We shall determine the set of
  lattice elements $x$ with $x \wedge a_m = \widehat 0$. There is only
  one solution to this equation with $x \leq a_m$ and that is $x =
  \widehat 0$. As the other solutions satisfy $x \nleq a_m$, they must
  have a block that contains $m$ and at least one other element. It
  follows that the solutions $x \neq \widehat 0$ are all elements of
  the form
  \begin{equation*} %%\label{eq:aperp}
    x = (x_1 \cdots x_tm)(\cdot) \cdots (\cdot) \text{\ with\ }
    \begin{cases}
      w(x_1) > s-w(m) & t=1 \\
      s \geq w(x_1)+ \cdots + w(x_t) > s-w(m) & t>1 
    \end{cases}
  \end{equation*}
  where all blocks but the unique block containing $m$ are singletons.
  There are $t+1$ elements in the block containing $m$ where $t$ is
  some number in the range $1 \leq t \leq s$. (All the solutions $x
  \neq \widehat 0$ are atoms in the lattice $L^s_m(w)$.) Since we are
  in a lattice, the \Mb\ function satisfies the equation
  \cite[Corollary 3.9.3]{stanley97}
  \begin{equation*}
    \mu^s_m(w)(\widehat 0, \widehat 1) = 
     -\sum_{\substack{ x \wedge a_m = \widehat 0 \\ x \neq \widehat
         0}} \mu^s_m(w)(x,\widehat 1)
  \end{equation*}
 which translates to
 %% wtrec in chrpol.prg
  \begin{multline}\label{eq:mblattice}
    \mu^s_m(w)(\widehat 0, \widehat 1)  = 
    -\sum_{\substack{ x \wedge a_m = \widehat 0 \\ x \neq \widehat 0}} 
    \mu^s_{|x|}(w(x))(\widehat 0, \widehat 1) = \\
    - \sum_{\substack{1 \leq x_1 \leq m-1 \\ w(x_1)>s-w(m)}}
    \mu^s_{m-1}(w(x_1m)w(\cdot)\cdots w(\cdot))(\widehat 0, \widehat 1)
    - \sum_{1 < t \leq s} 
   \sum_{\substack{1 \leq x_1,\ldots,x_t \leq m-1 \\ 
    s \geq w(x_1)+\cdots+w(x_t) > s-w(m)}} 
   \mu^s_{m-t}(w(x_1\cdots
   x_tm))w(\cdot)\cdots w(\cdot))(\widehat 0, \widehat 1) 
  \end{multline}
  This describes a recursive procedure for computing all values of the
  \Mb\ function on the weight lattices $L^s_m(w)$.  

  As an illustration we compute $\mu^2_6(1^6)(\widehat 0, \widehat
  1)$. Using \eqref{eq:mblattice} twice gives
%  wtrec in chrpol.prg
  \begin{equation*}
   \mu^2_6(1^6)(\widehat 0, \widehat 1) = 
   -10 \mu^2_4(3111)(\widehat 0, \widehat 1) =
   10( \mu^2_3(411)(\widehat 0, \widehat 1) + \mu^2_2(33)(\widehat 0,
   \widehat 1) )
  \end{equation*}
  The lattices $L^2_4(411)$ and $L^2_2(33)$
   have $4$ and $2$ elements, respectively, and they look like
   
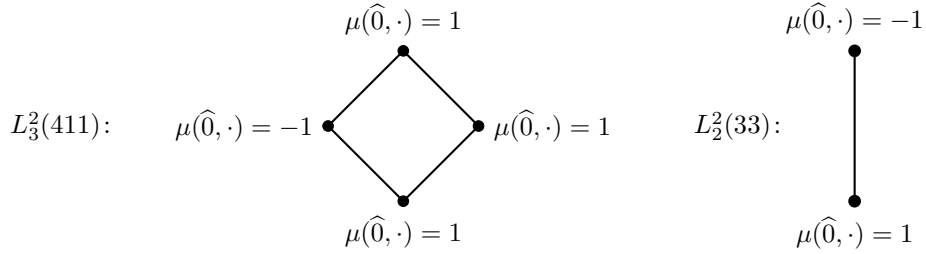
\begin{center}
     \begin{tikzpicture}
  [vertex/.style= {shape=circle,  fill=#1!100, minimum size =
  4pt, inner sep =0pt,draw}, vertex/.default=black]
  
  \node at (-4.5,1) {$L^2_3(411) \colon$};
  \draw [thick] (0,0)  -- (-1,1) -- (0,2) -- (1,1)  -- (0,0) -- cycle;
  \node at (0,0) [vertex, label = below:{$\mu(\widehat 0, \cdot)=1$}] {}; 
  \node at (-1,1) [vertex, label = left:{$\mu(\widehat 0, \cdot)=-1$}] {};
  \node at (0,2) [vertex, label = above:{$\mu(\widehat 0, \cdot)=1$}] {}; 
  \node at (1,1) [vertex,label = right:{$\mu(\widehat 0, \cdot)=1$}] {};

  \node at (4.5,1) {$L^2_2(33) \colon$};
  \draw [thick] (6,0) node[vertex,label = below:{$\mu(\widehat 0,
    \cdot)=1$}] {} --  
  (6,2) node[vertex,label = above:{$\mu(\widehat 0, \cdot)=-1$}] {};
\end{tikzpicture}
\end{center}
so that $\mu^2_3(411)(\widehat 0, \widehat 1) = 1$ and
$\mu^2_2(33)(\widehat 0, \widehat 1) = -1$. Therefore
$\mu^2_6(1^6)(\widehat 0, \widehat 1)=0$. 
 % Table~\ref{tab:mum2} shows $\mu^2_m(w)(\widehat 0, \widehat 1)$ where
% $w$ is the uniform weight $w=1^m$ with constant value $1$ or
% $w=3^11^{m-1}$ which has value $3$ on one element and value $1$ on the
% rest of the elements.  
\end{rmk}

We remind the reader of the well-known fact that $\mu_m^s(w)(\widehat
0, \widehat 1)$ is the reduced \euc\ of the open interval 
$L^s_m(w)(\widehat 0, \widehat 1)$ between $\widehat 0$ and $\widehat
1$  in the lattice $L^s_m(w)$.

\begin{prop}\cite[\S 6]{rota64} \cite[Proposition 3.8.5]{stanley97}
  Let $x<y$ be two elements in a finite poset. The value of the \Mb\
  function on the pair $(x,y)$ is the reduced \euc\ of the open
  interval $(x,y)$.
\end{prop}
\begin{proof}
  Write $\mu$ be the \Mb\ function of $P$ and $\mathrm{E}$ for \euc
  . The closed interval from $x$ to $y$ has \euc\ $1$ since it has a
  smallest element. Thus
  \begin{multline*}
    1 = \mathrm{E}([x,y]) = \sum_{a,b \in [x,y]} \mu(a,b) =
     \sum_{a,b \in (x,y)} \mu(a,b) + \sum_{a \in [x,y]} \mu(a,y) +
      \sum_{b \in [x,y]} \mu(x,b) - \mu(x,y) \\
      = \mathrm{E}((x,y)) + 0 + 0 -\mu(x,y) = \mathrm{E}((x,y))-\mu(x,y)
  \end{multline*}
  or $\mu(x,y) = \widetilde{\mathrm{E}}((x,y))$.
\end{proof}

For $1 \leq s \leq m+1$ let $B(m,s)$ be the graded poset of 
nonempty subsets of $[m]$ of
cardinality less than $s$. 

\begin{lemma}\label{lemma:Bms}
  The reduced \euc\ of $B(m,s)$ is
  \begin{equation*}
    \widetilde E(B(m,s)) = (-1)^s \binom {m-1}{s-1},  \qquad
    1 \leq s\leq m+1
  \end{equation*}
\end{lemma}
\begin{proof}
   It is rather easy to get the recurrence relation
   \begin{align*}
    &E(B(m,2)) = m \\
    &E(B(m,s)) = 
   E(B(m,s-1)) + \binom m{s-1} \sum_{j=1}^{s-1} (-1)^{s-1-j}
   \binom{s-1}j, \qquad 2<s<2+m
  \end{align*}
  Since the sum of binomial coefficients has value $(-1)^s$, we
  get the recurrence relation
  \begin{align*}
     &\widetilde E(B(m,2)) = m-1 \\
    &\widetilde E(B(m,s)) = 
   \widetilde E(B(m,s-1)) + (-1)^s\binom m{s-1},  \qquad 2<s<2+m
  \end{align*}
  for the reduced \euc . The claim of the lemma follows immediately. 
\end{proof}

\begin{exmp}[Reduced \euc s of the $s$-chromatic lattice intervals
  $L^s_m(w)(\widehat 0, \widehat 1)$]
  The reduced \euc s $\mu_m^s(1^m)(\widehat 0, \widehat 1) =
  \widetilde E(L^s_m(1^m)(\widehat 0, \widehat 1))$, $m \geq s+2$, for
  $s=1,2,\ldots,8$ are
%% [mu0(m,2) : m in [1..20]]; in chrpol.prg
\begin{align*}
%% 1
&2, -6, 24, -120, 720, -5040, 40320, -362880, 3628800, -39916800,
479001600, -6227020800, 87178291200, \ldots \\
%% 2
& 3, -6, 0, 90, -630, 2520, 0, -113400, 1247400, -7484400, 0,
681080400, -10216206000, 81729648000, \ldots \\
%% 3
& 4, -10, 20, -70, 560, -4200, 25200, -138600, 924000, -8408400,
84084000, -798798000, 7399392000, \ldots \\
%% 4
& 5, -15, 35, -70, 0, 2100, -23100, 173250, -1051050, 5255250,
-15765750, -105105000, 2858856000,\ldots \\
% -38594556000, 407386980000,\ldots \\
%% 5
& 6, -21, 56, -126, 252, -924, 11088, -126126, 1093092,
-7693686, 46414368, -254438184, 1492322832, \ldots \\ % -12710473776\ldots \\
%%6
& 7, -28, 84, -210, 462, -924, 0, 42042, -630630,
6390384, -51459408, 351639288, -2118412296,11406835440\ldots \\
%% 7
&   8, -36, 120, -330, 792, -1716, 3432, -12870,
205920, -3150576, 35706528,-322583976,2460949920 \ldots \\
%% 8
&  9, -45, 165, -495, 1287, -3003, 6435, -12870, 0,
787644, -14965236,191222460, -1920538620 \ldots
\end{align*}
The first sequence, $\mu_m^1(1^m)(\widehat 0,
\widehat 1)$, $m \geq 2$, is the sequence $(-1)^{m-1}(m-1)!$ of
reduced \euc s of the lattice of partitions of $[m]$ \cite[Example
3.10.4]{stanley97}.  The second sequence, $\mu_m^2(1^m)(\widehat 0,
\widehat 1)$, $m \geq 3$, seems to coincide with first terms of the
sequence \href{http://oeis.org/A009014}{A009014} from The On-Line
Encyclopedia of Integer Sequences (OES). The remaining $6$ sequences
apparently do not match any sequences of the OES.

The first $s$ terms of these sequences are signed binomial
coefficients. This is because the interval $(\widehat 0, \widehat 1)$
in $L^s(D[m])$ is isomorphic to the opposite of the poset $B(m,m-s)$
when $s+2 \leq m \leq 2s+1$. Thus the reduced \euc\
\begin{equation*}
  \mu_m^s(1^m)(\widehat 0, \widehat 1) = \widetilde E (B(m,m-s)) =
  (-1)^{m-s} \binom {m-1}s, \qquad s+2 \leq m \leq 2s+1,
\end{equation*}
according to Lemma~\ref{lemma:Bms}.

The first terms of the sequence $\mu_m^2(3^11^{m-1})(\widehat 0,
\widehat 1)$, $m \geq 3$, of reduced \euc s of the weighted lattice intervals
$L^2_m(3^11^{m-1})(\widehat 0,
\widehat 1)$, %% wtmu(wt1(4,2),2) L^2_4
%%%%%%%%%%%%
% mu1:=func<m,s |
% > &+[(-1)^(i-1)*Factorial(i-1)*St({1..m},wt1(m,2),i,s) : i in [Ceiling(m/s)..m\
% ]]>;
% > [mu1(m,2) : m in [3..9]];
%%%%%%%%%%%%%%%%%%%%%%%%%%
\begin{equation*}
  1, 0, -6, 30, -90, 0, 2520, -22680, 113400, 0, -7484400, 97297200,
  -681080400, 0, 81729648000, -1389404016000, \ldots
\end{equation*}
seem to coincide up to sign with first terms of the sequence
\href{http://oeis.org/A009775}{A009775} from OES.  The sequence of
reduced \euc s $\mu_m^2(3^21^{m-2})(\widehat 0, \widehat 1)$, $m \geq
3$, of the lattice interval $L^2_m(3^21^{m-2})(\widehat 0, \widehat
1)$,
\begin{multline*}
2, -4, 6, 6, -120, 720, -2520, -2520, 136080, -1360800, 7484400,
7484400, \\ -778377600, 10897286400, -81729648000, -81729648000, 
 13894040160000, \ldots
\end{multline*}
apparently does not match any sequence in the OES.
%%%%%%%%%%%%%%  from chrpol.prg
% mu2:=func<m,s | 
% &+[(-1)^(i-1)*Factorial(i-1)*St({1..m},wt2(m,2),i,s) : 
% i in [Ceiling(m/s)..m]]>;
% [mu2(m,2) : m in [3..9]];
%%%%%%%%%%%%%%
% $w=2^13^11^{m-2}$, $m \geq 3$, $\mu^2_m(w)$,
% \begin{equation*}
%   2, -4, 6, 6, -120, 720, -2520,  -2520, 136080, -1360800, 
%   7484400, 7484400, \ldots
% \end{equation*}
% as before 
\end{exmp}

% \begin{exmp}
% For instance, for
% the $2$-chromatic lattice for the $4$-simplex $D[5]$ we have the
% following graphical presentation of $L^2(D[5])$
% %\begin{figure}[t]
% %  \centering
% \begin{equation*}
%   \xy  
%   {\xypolygon5"B"{ ~:{(16.5,0):}   ~={0}    ~<{} ~>{} @{*} }},
%   {\xypolygon10"A"{~:{(6.0,0):}  ~={-18} ~<{-}   ~>{} @{*} }},
%     "B0", *@{*};
%   "A1"; "B1" **\dir{-};
%    "A2"; "B1" **\dir{-};
%    "A3"; "B1" **\dir{-};
%    "A10"; "B1" **\dir{-};
%    %%%%%%
%   "A2"; "B2" **\dir{-};
%   "A3"; "B2" **\dir{-};
%   "A4"; "B2" **\dir{-};
%   "A5"; "B2" **\dir{-};
%   %%%%%%
%   "A4"; "B3" **\dir{-};
%   "A5"; "B3" **\dir{-};
%   "A6"; "B3" **\dir{-};
%   "A7"; "B3" **\dir{-};
%   %%%%%%
%   "A6"; "B4" **\dir{-};
%   "A7"; "B4" **\dir{-};
%   "A8"; "B4" **\dir{-};
%   "A9"; "B4" **\dir{-};
%   %%%%%%
%   "A8"; "B5" **\dir{-};
%   "A9"; "B5" **\dir{-};
%   "A10"; "B5" **\dir{-};
%   "A1"; "B5" **\dir{-};
% \endxy
% \end{equation*}
% %\caption{The $2$-chromatic lattice $L^2(D[5])-\{\widehat 1\}$ of
% %  $D[5]$}
% %\end{figure}
% where we have left out the maximal element $\widehat 1$. The
% convention here is that $y$ covers $x$ if $x$ and $y$ are connected by
% an edge with $y$ farther from the center than $x$.
% \end{exmp}

Define the {\em $s$-monochrome set\/} of a map $\mathrm{col} \colon
V(K) \to [r]=\{1,\ldots,r\}$ to be the set
\begin{equation*}
  M^s(\mathrm{col}) = \{ \sigma \in F^s(K) \mid |\mathrm{col}(\sigma)| = 1 \}
\end{equation*}
of all monochrome $s$-simplices in $K$.  The map $\mathrm{col}$ is an
$(r,s)$-coloring of $K$ if and only if
$M^s(\mathrm{col})=\emptyset$.

\begin{lemma}\label{lemma:Mscol}
  The $s$-monochrome set $M^s(\mathrm{col})$ of any map $\mathrm{col}
  \colon V(K) \to [r]$ is an element of the $s$-chromatic lattice
  $L^s(K)$.
\end{lemma}
\begin{proof}
  Let $S$ be a connected component of $M^s(\mathrm{col})$. Since $S$ is
  connected, all vertices in $S$ have the {\em same\/} color. Let
  $\sigma \in F^s(K)$ be an $s$-simplex of $K$ such that $\sigma
  \subset V(S)$. The $\sigma$ is monochrome: $\sigma \in
  M^s(\mathrm{col})$. By Lemma~\ref{lemma:closedcomp}, $S$ is closed.
\end{proof}

%% chrpol in chrpol.prg
\begin{thm}\label{thm:LsK}
  The number of $(r,s)$-colorings of $K$ is
  \begin{equation*}
    \chrps Krs=\sum_{T \in L^s(K)} \mu(\widehat 0,T)  r^{|\pi(T)|}
  \end{equation*}
 where $\mu$ the \Mb\ function for the $s$-chromatic lattice $L^s(K)$.
 %  More generally,  for any element $S$ of $L^s(K)$, 
%   \begin{equation*}
%     \chi_S(K,r,s)=\sum_{\substack{T \in L^s(K) \\ S \leq T}}
%     \mu(S,T)  r^{|\pi(T)|}
%   \end{equation*}
%   is the number of maps $\mathrm{col} \colon K \to
%   \{1,\ldots,r\}$ with monochrome set of $s$-simplices
%   $M^s(\mathrm{col})=S$.
\end{thm}
\begin{proof}
  For any $B \in L^s(K)$, let $\chi(K,r,s,B)$ be the number of maps
  $\mathrm{col} \colon V(K) \to [r]$ with $M^s(\mathrm{col})=B$.  We
  want to determine $\chi(K,r,s,\emptyset) = \chrps sKr$.  For any $A
  \in L^s(K)$,
  \begin{equation*}
     r^{|\pi(A)|} =
     \sum_{A \leq B} \chi(K,r,s,B)                       
  \end{equation*}
  because there are $r^{|\pi_0(A)|} r^{m(K)-|V(A)|} =
  r^{|\pi(A)|}$ maps $\mathrm{col} \colon V(K) \to [r]$ with $A
  \leq M^s(\mathrm{col})$.  Equivalently,
  \begin{equation*}
    \sum_{A \leq B} \mu(A,B)  r^{|\pi(B)|} = \chi(K,r,s,A)
  \end{equation*}
  by \Mb\ inversion \cite[Proposition 3.7.1]{stanley97}. The statement
  of the theorem is the particular case of this formula where
  $A=\widehat 0$.
\end{proof}

The  defining rules for the \Mb\ function of the poset $L^s(K)$
\cite[3.7]{stanley97}  
\begin{itemize}
\item $\mu(S,S)=1$ for all $S \in L^s(K)$
\item $\sum_{R \leq S \leq T} \mu(R,S)=0$ when $R \lneqq T$
\item $\mu(R,S)=0$ when $R \nleq S$
\end{itemize}
imply that $\mu(\widehat 0,\widehat 0)=1$ and
$\mu(\widehat 0,\{\sigma\})=-1$ for every $s$-simplex $\sigma \in
F^s(K)$.

\begin{cor}\label{cor:LsK}
  The highest degree terms of the $s$-chromatic polynomial are
  \begin{equation*}
    \chrps Krs = r^{m(K)} - f_s(K)r^{m(K)-s} + \cdots 
  \end{equation*}
  Thus the  $s$-chromatic polynomial determines $f_0(K)$ and $f_s(K)$.
\end{cor}
\begin{proof}
  The $s$-chromatic polynomial is
  \begin{equation*}
    \chrps Krs = \mu(\widehat 0,\widehat 0)r^{f_0(K)} +
        \sum_{\sigma \in F^s(K)}\mu(\widehat 0,\{\sigma\})r^{f_0(K)-s}
        + \cdots
  \end{equation*}
  where $\mu(\widehat 0,\widehat 0)=1$ and  $\mu(\widehat
  0,\{\sigma\})=-1$ for all $s$-simplices $\sigma$ of $K$.
\end{proof}

\begin{example}
  Consider the $2$-dimensional complex $K$ from Example~\ref{exmp:K}.
%%%%%%%%%%%%%%%%%%%%%%%%%%%%%%%
% \begin{equation*}
%   \xy 0;/r.16pc/:
%    {\xypolygon3"B"{ ~:{(16.5,0):}   ~={0}    ~<{} ~>{-} @{*} }},-(16.5,0),
%    {\xypolygon3"A"{ ~:{(-16.5,0):}   ~={0}    ~<{} ~>{-} @{*}}};
%    "A1"+(-5,0)*{1};
%    "A3"+(0,5)*{2};
%    "A2"+(0,-5)*{3};
%    "B1"+(0,-5)*{4};
%     "B1",+(16.5,0), 
%    {\xypolygon3"C"{ ~:{(-16.5,0):}   ~={0}    ~<{} ~>{-} @{*}}};
%    "C2"+(0,-5)*{5};
%    "C3"+(0,5)*{6};
%   \endxy
% \end{equation*}
%%%%%%%%%%%%%%%%
%  \begin{equation*}
%  \xy 0;/r.20pc/:
%  (0,0)*{1}="A";
%  (20,0)*{2}="B";
%  (40,0)*{4}="C";
%  (80,0)*{6}="E";
%  (20,20)*{3}="F";
%  (60,20)*{5}="G";
%   "A"; "B" **\dir{-};
%   "A"; "E" **\dir{-};
%   "B"; "E" **\dir{-};
%   "B"; "F" **\dir{-};
%   "C"; "F" **\dir{-};
%   "A"; "B" **\dir{-};
%   "C"; "G" **\dir{-};
%   "E"; "G" **\dir{-};
%   "A"; "F" **\dir{-};
%  \endxy
% \end{equation*}
%%%%%%%%%%%%%%%%%%%
%The subset $S=\{\{1,2,3\},\{4,5,6\} \}$ of $F^2(K)$ is not
%closed but both its two connected components, $\{\{1,2,3\}\}$ and
%$\{\{4,5,6\}\}$, are closed. Thus $S$ is a member of $L^2(K)$. In fact,
The $2$-chromatic lattice $L^2(K)$ of $K$
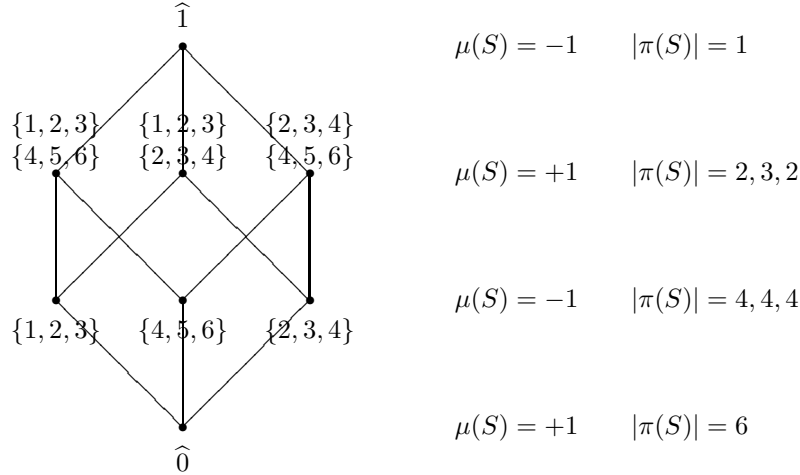
\begin{equation*}
   \xy 0;/r.20pc/:
  (0,0)*@{*}="1"+(0,-5)*{\{1,2,3\}};
  (20,0)*@{*}="2"+(0,-5)*{\{4,5,6\}};
  (40,0)*@{*}="3"+(0,-5)*{\{2,3,4\}};
  (0,20)*@{*}="4"+(0,+5)*{
    \begin{array}{c}
      \{1,2,3\} \\ \{4,5,6\}
    \end{array}};
  (20,20)*@{*}="5"+(0,+5)*{
    \begin{array}{c}
      \{1,2,3\} \\ \{2,3,4\}
    \end{array}};
  (40,20)*@{*}="6"+(0,+5)*{
    \begin{array}{c}
      \{2,3,4\} \\ \{4,5,6\}
    \end{array}};
  (20,-20)*@{*}="0"+(0,-5)*{\widehat 0};
  (20,40)*@{*}="7"+(0,+5)*{\widehat 1};
   "1"; "4" **\dir{-};
   "1"; "5" **\dir{-};
   "2"; "4" **\dir{-};
   "2"; "6" **\dir{-};
   "3"; "5" **\dir{-};
   "3"; "6" **\dir{-};
   "0"; "1" **\dir{-};
   "0"; "2" **\dir{-};
   "0"; "3" **\dir{-};
   "7"; "4" **\dir{-};
   "7"; "5" **\dir{-};
   "7"; "6" **\dir{-};
   (90,-20)*{\mu(S)=+1 \qquad |\pi(S)| = 6\phantom{,4,4} };
   (90,0)*{\mu(S)=-1 \qquad |\pi(S)| = 4,4,4};
   (90,20)*{\mu(S)=+1 \qquad |\pi(S)| = 2,3,2};
   (90,40)*{\mu(S)=-1 \qquad |\pi(S)| = 1\phantom{,4,4}};
  \endxy 
\end{equation*}
%   \begin{equation*}
%     \xymatrix{
%       & {(-1,1)} \\
%       {(1,3)} \ar@{-}[ur] & {(1,2)} \ar@{-}[u] & {(1,2)}
%       \ar@{-}[ul] \\
%       {(-1,4)} \ar@{-}[u] \ar@{-}[ur] & 
%       {(-1,4)} \ar@{-}[ul] \ar@{-}[ur] &
%       {(-1,4)} \ar@{-}[u] \ar@{-}[ul] \\
%      & {(1,6)} \ar@{-}[ul] \ar@{-}[u] \ar@{-}[ur]}
%   \end{equation*}
consists of {\em all\/} subsets of $F^2(K)$.  The $2$-chromatic
polynomial is
  \begin{equation*}
    \chrps Kr2 = r^6 - r^4 - r^4 -r^4 + r^2 + r^3 + r^2 - r = 
           r^6 - 3r^4 + r^3 + 2r^2 -r 
                %= 15[r]_2+73[r]_3+62[r]_4+15[r]_5+[r]_6 
  \end{equation*}
  $K$ has $\chrps K22 = 30$ $(2,2)$-colorings and $\chrps K32 = 528$
  $(3,2)$-colorings. 
%  The \euc\ of $L^2(K)$ is $E(L^2(K)) =
%  1+\mu(\widehat 0, \widehat 1)= 0$ \cite[p 346]{rota64}.
\end{example}

% \begin{equation*}
%  \xy /l4.5pc/:
%      {\xypolygon4"B"{~:{(1.41,0):}~<<{}~>{}}{}}, 
%       {\xypolygon4"C"{~={0}~<<{}~>{}}{}},
%      "B4"+(0,-0.2)*{\alpha_1}, 
%      "B3"+(0,-0.2)*{\alpha_1+2\alpha_2},
%      "C3"+(-0.2,0)*{\alpha_2},
%      "C4"+(0,-0.2)*{\alpha_1+\alpha_2},
%      \endxy
%  \end{equation*}

% The lattice $L^s(K)$ is a geometric lattice when $s=1$ \cite[\S9,
% Lemma~1]{rota64}. The following example shows that $L^s(K)$ need not
% be geometric when when $s>1$.

\begin{exmp}\label{exmp:MB}
%%/home/moller/underv/algtop/magma/chrpol.prg procedure chrpol
%%/home/moller/underv/algtop/magma/LMB.log
The triangulation $\mathrm{MB}$ of the \Mb\ band with $f$-vector
$f(\mathrm{MB}) = (5,10,5)$ shown in Figure~\ref{fig:MBcol}
% \begin{equation*}
%  \xy 0;/r.10pc/:
%  (0,0)*@{*}="A"+(-4,-5)*{1};
%  (40,0)*@{*}="B"+(0,-5)*{2};
%  (80,0)*@{*}="C"+(0,-5)*{3};
%  (120,0)*@{*}="D"+(0,-5)*{4};
%  (20,20)*@{*}="E"+(0,5)*{4};
%  (60,20)*@{*}="F"+(0,5)*{5};
%  (100,20)*@{*}="G"+(0,5)*{1};
%   "A"; "B" **\dir{-};
%   "A"; "E" **\dir{-};
%   "B"; "E" **\dir{-};
%   "B"; "F" **\dir{-};
%   "E"; "F" **\dir{-};
%   "B"; "C" **\dir{-};
%   "E"; "F" **\dir{-};
%   "C"; "F" **\dir{-};
%   "F"; "G" **\dir{-};
%   "E"; "G" **\dir{-};
%   "D"; "G" **\dir{-};
%   "C"; "G" **\dir{-};
%   "C"; "D" **\dir{-};
%  \endxy
% \end{equation*}
%%%%%%%%%%%%%%%%%%%
%   \begin{equation*}
%    \xy 0;/r.20pc/:
%   (0,0)*{2}="2";
%   (20,0)*{4}="4";
%   (40,0)*{3}="3";
%   (60,0)*{5}="5";
%   (80,0)*{6}="6";
%   %%%%%%%%%%%%
%   (0,20)*{11}="11";
%   (20,20)*{10}="10";
%   (40,20)*{8}="8";
%   (60,20)*{9}="9";
%   (80,20)*{7}="7";
%   %%%%%%%%
%    "5"; "7" **\dir{-};
%    "6"; "7" **\dir{-};
%     %%%%%%%%
%    "4"; "8" **\dir{-};
%    "5"; "8" **\dir{-};
%     %%%%%%%%
%    "3"; "9" **\dir{-};
%    "6"; "9" **\dir{-};
%     %%%%%%%%
%    "2"; "10" **\dir{-};
%    "3"; "10" **\dir{-};
%     %%%%%%%%
%    "2"; "11" **\dir{-};
%    "4"; "11" **\dir{-};
%   \endxy
%   \end{equation*}
has the following (reduced) $2$-chromatic lattice
$L^2(\mathrm{MB})-\{\widehat 0, \widehat 1\}$ 
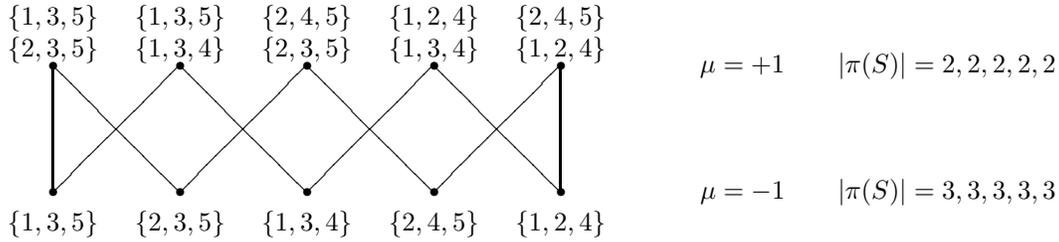
  \begin{equation*}
   \xy 0;/r.20pc/:
  (0,0)*@{*}="2"+(0,-5)*{\{1,3,5\}};
  (20,0)*@{*}="4"+(0,-5)*{\{2,3,5\}};
  (40,0)*@{*}="3"+(0,-5)*{\{1,3,4\}};
  (60,0)*@{*}="5"+(0,-5)*{\{2,4,5\}};
  (80,0)*@{*}="6"+(0,-5)*{\{1,2,4\}};
  %%%%%%%%%%%%
  (0,20)*@{*}="11"+(0,5)*{
    \begin{array}[c]{c}
      \{1,3,5\} \\ \{2,3,5\}
    \end{array}};
  (20,20)*@{*}="10"+(0,5)*{
    \begin{array}[c]{c}
      \{1,3,5\} \\ \{1,3,4\}
    \end{array}};
  (40,20)*@{*}="8"+(0,5)*{
    \begin{array}[c]{c}
      \{2,4,5\} \\ \{2,3,5\}
    \end{array}};
  (60,20)*@{*}="9"+(0,5)*{
    \begin{array}[c]{c}
      \{1,2,4\} \\ \{1,3,4\}
    \end{array}};
  (80,20)*@{*}="7"+(0,5)*{
    \begin{array}[c]{c}
      \{2,4,5\} \\ \{1,2,4\}
    \end{array}};
  %%%%%%%%
   "5"; "7" **\dir{-};
   "6"; "7" **\dir{-};
    %%%%%%%%
   "4"; "8" **\dir{-};
   "5"; "8" **\dir{-};
    %%%%%%%%
   "3"; "9" **\dir{-};
   "6"; "9" **\dir{-};
    %%%%%%%%
   "2"; "10" **\dir{-};
   "3"; "10" **\dir{-};
    %%%%%%%%
   "2"; "11" **\dir{-};
   "4"; "11" **\dir{-};
   (130,0)*{\mu=-1 \qquad |\pi(S)| = 3,3,3,3,3};
   (130,20)*{\mu=+1 \qquad |\pi(S)| = 2,2,2,2,2};
  \endxy
  \end{equation*}
  and $2$-chromatic polynomial
  \begin{equation*}
    \chrps {\mathrm{MB}}r2 = r^5 - 5r^3 + 5r^2 -r 
    % =5[r]_2 + 20[r]_3 + 10[r]_4 + [r]_5
  \end{equation*} 
  The lattice $L^2(\mathrm{MB})$ is graded but it is still not
  semi-modular \cite[Proposition 3.3.2]{stanley97}: The meet and join
  of $a=\{\{2,3,5\}\}$ and $b=\{\{1,3,4\}\}$ are $a \wedge b =
  \widehat 0$ and $a \vee b = \widehat 1$. Thus $a$ and $b$ cover $a
  \wedge b$ but $a \vee b$ covers neither $a$ nor $b$.  
%  The \euc\ of
%  $L^2(\mathrm{MB})$ is $E(L^2(\mathrm{MB})) = 1+\mu(\widehat 0,
%  \widehat 1)= 0$. %% chrpol.prg procedure chrpol
%% /home/moller/underv/algtop/magma/chrpol.prg : IsGraded(L)
\end{exmp}

\begin{figure}[t]
  \centering
\begin{tikzpicture} %%P2
  [vertex/.style= {shape=circle,  fill=#1!100, minimum size =
  8pt, inner sep =0pt,draw}, vertex/.default=black, scale=1.25]

\begin{scope}[xshift=-4.5cm,yshift=1.75cm]
  
\draw (0,0) circle (2cm);

\coordinate [vertex=blue, label=right:$1$] (A1) at (0:2cm);
\coordinate [vertex=red, label=above right:$2$] (A2) at (60:2cm);
\coordinate [vertex=yellow, label=above left:$3$] (A3) at (120:2cm);

\coordinate [vertex=blue, label=left:$1$] (B1) at (-180:2cm);
\coordinate [vertex=red, label=below left:$2$] (B2) at (-120:2cm);
\coordinate [vertex=yellow, label=below right:$3$] (B3) at (-60:2cm);

\begin{scope}[rotate=-60]
\coordinate [vertex=blue, label=0:$4$] (C4) at (240:1cm);
\coordinate [vertex=red, label=266:$5$] (C5) at (120:1cm);
\coordinate [vertex=yellow, label=100:$6$] (C6) at (0:1cm);
\end{scope}

%\draw (2,0) arc (2:58:2cm);

\draw (A1) -- (C5) -- (C6) -- (A1) -- cycle;
\draw (A2) -- (C5) -- (A1);
\draw (A2) -- (C5) -- (A3);
\draw (A3) -- (C4) -- (C5) -- (A3);
\draw (A3) -- (C4) -- (B1);
\draw (B2) -- (C4) -- (C6) -- (B2);
\draw (B2) -- (C6) -- (B3);
\end{scope}

\begin{scope}[scale=1.2]  %% torus
\coordinate [vertex=red, label=below:$1$] (A1) at (-.75,0);
\coordinate [vertex=blue, label=below:$4$] (A2) at (.75,0);
\coordinate [vertex=yellow, label=below:$7$] (A3) at (2.25,0);
\coordinate [vertex=red, label=below:$3$] (A4) at (3.75,0);

\coordinate [vertex=red, label=left:$3$] (B2) at (0,1);
\coordinate [vertex=red, label=left:$2$] (B3) at (0,2);
\coordinate [vertex=red, label=above:$1$] (B4) at (-.75,3);

\coordinate [vertex=blue, label=above:$4$] (C2) at (.75,3);
\coordinate [vertex=yellow, label=above:$7$] (C3) at (2.25,3);
\coordinate [vertex=red, label=above:$3$] (C4) at (3.75,3);

\coordinate [vertex=red, label=right:$2$] (D2) at (3,1);
\coordinate [vertex=red, label=right:$1$] (D3) at (3,2);

\coordinate [vertex=blue, label=above:$5$] (E1) at (1.5,2);
\coordinate [vertex=blue, label=below:$6$] (E2) at (1.5,1);

\draw (A1) -- (A2) -- (A3) -- (A4) -- (D2) -- (D3) -- (C4) -- (C3) --
(C2) -- (B4) -- (B3) -- (B2) -- (A1);

\draw (C2) -- (C3) -- (E1) -- (C2);
\draw (A2) -- (A3) -- (E2) -- (A2);
\draw (D3) -- (E1) -- (E2) -- (D3);
\draw (D2) -- (E2) -- (A3) -- (D2);
\draw (C3) -- (D3);
\draw (B3) -- (C2);
\draw (B3) -- (E1);
\draw (B2) -- (E1);
\draw (B2) -- (E2);
\draw (B2) -- (A2);
\draw [dotted,<->,gray] (A4) .. controls (4.25,0.5) and (4.25,2.5) .. (C4);
\draw [dotted,<->,gray] (A1) .. controls (-1.25,0.5) and (-1.25,2.5)..  (B4);
%\draw [dotted,<->,gray] (A2) -- (C2);
%\draw [dotted,<->,gray] (A3) -- (C3);
\end{scope}
\end{tikzpicture}
\caption{$(3,2)$-colorings of $\mathrm{P2}$ and $\mathrm{MT}$}
  \label{fig:P2col}
\end{figure}

\begin{exmp}\label{exmp:MT}
 %  For the \Mb\ band $\mathrm{MB}$ with $f$-vector $(5, 10, 5 )$ we get
%   \begin{equation*}
%     \chrps {\mathrm{MB}}r1 = r^5 - 10r^4 + 35r^3 - 50r^2 + 24r, \qquad
%     \chrps {\mathrm{MB}}r2 = r^5 - 5r^3 + 5r^2 - r
%   \end{equation*}
  Let $\mathrm{MT}$ be \Mb 's minimal triangulation of the torus with
  $f$-vector $f(\mathrm{MT})=(7, 21, 14 )$ and $\mathrm{P}2$ the
  triangulation of the projective plane with $f$-vector
  $f(\mathrm{P}2) = (1,6,15,10)$ shown in Figure~\ref{fig:P2col}
  (decorated with $(3,2)$-colorings).  The \chp s of these two
  simplicial complexes are
  \begin{alignat*}{3}
    &\chrps {\mathrm{MT}}r1 = [r]_7, & &\qquad &
    &\chrps {\mathrm{MT}}r2 = r^7 - 14r^5 + 21r^4 + 7r^3 - 21r^2 + 6r
    % = 84[r]_3 +  231[r]_4 +  126[r]_5 + 21[r]_6 + [r]_7 \\
    \\
    &\chrps {\mathrm{P}2}r1 = [r]_6, & &\qquad &
    & \chrps {\mathrm{P}2}r2 = r^6 - 10r^4 + 15r^3 - 6r^2
  \end{alignat*}
  In both cases, the $1$-skeleton is the complete graph on the vertex
  set. The chromatic numbers are $\chs 1{\mathrm{MT}} = 7$, $\chs
  1{\mathrm{P}2} = 6$, and $\chs 2{\mathrm{MT}} = 3 = \chs
  2{\mathrm{P}2}$.  
 % The \euc\ is $E(L^2(K)) = 1+\mu(\widehat 0, \widehat 1)= 7$.
%%%%%%%%%%%%%
%
%  Figure~\ref{fig:P2col} also shows a $(3,2)$-coloring of a
%  triangulation $\mathrm{P}2$, $f=(1, 6, 15, 10)$, of the projective
%  plane. The \chp s of $\mathrm{P}2$ are
%  \begin{equation*}
%    \chi(\mathrm{P}2,r,1) = [r]_6, \qquad
%    \chi(\mathrm{P}2,r,2) = r^6 - 10r^4 + 15r^3 - 6r^2 = 
%    r(r-1)(r-2)(r^2 + 3r - 3)
%  \end{equation*}
% and the $2$-\chp\ again has a real root between $0$ and $1$.
%% [Evaluate(r^6 - 10*r^4 + 15*r^3 - 6*r^2,x) : x in [i/(i+1) : i in
 %% [0..10]]]; 
%   The
%   coefficients of the $2$-\chp\ $\chi(\mathrm{MT},r,2)$ are not
%   sign-alternating and they are not $\log$-concave
%   (Definition~\ref{defn:logconcave}) in absolute value. In contrast,
%   the coefficients of a $1$-\chp\ are always sign-alternating and
%   $\log$-concave \cite{huh2012}.
%  (Since the coefficients are not sign-alternating we see,
%  again, that $L^2(\mathrm{MT})$ is not a geometric lattice )
 %  The $1$-\chp\ of a simple graph has no negative roots, since the
%   coefficients are sign-alternating, and no roots between $0$ and $1$
%   \cite{woodall77}.  This is not true for $s$-\chp s in general as the
%   $2$-\chp\ for $\mathrm{MT}$
\end{exmp}

  The \chp s of  simple graphs (the $1$-\chp s of simplicial complexes)
  are known to have these properties:
  \begin{itemize}
  \item The coefficients are sign-alternating \cite[\S7,
    Corollary]{rota64}
  \item  The coefficients are $\log$-concave
  (Definition~\ref{defn:logconcave}) in absolute value \cite{huh2012}
\item There are no negative roots and no roots between $0$ and $1$
  \cite{woodall77} 
  \end{itemize}
  In contrast, the coefficients of the $2$-\chp\
  \begin{equation*}
     \chrps {\mathrm{MT}}r2 =  r^7 - 14r^5 + 21r^4 + 7r^3 - 21r^2 + 6r =
     [r]_3(r+1)(r^3 + 2r^2 - 9r + 3)
  \end{equation*}
  are not sign-alternating, not $\log$-concave in absolute value, and
  the polynomial has a negative root and a root between $0$ and $1$.

\subsection{The $s$-\chp\ in falling factorial form}
\label{sec:FF} 

Theorem~\ref{lemma:StK} provides an interpretation of the coefficients
of the falling factorial $[r]_i$ in the $s$-chromatic polynomial of
the simplicial complex $K$.

\begin{defn}\label{defn:SKrs}
$S(K,r,s)$ is the number of partitions of $V(K)$ into $r$
$s$-independent blocks. % (Definition~\ref{defn:independent}).  
\end{defn}

We think of $S(K,r,s)$ as an $s$-Stirling number of the second kind
for the simplicial complex $K$. If $s > \dim(K)$, then there are no
$s$-simplices in $K$ and all partitions of $V(K)$ are $s$-independent,
so that $S(K,r,s)$ is the Stirling number of the second kind
$S(m(K),r)$ \cite[p $33$]{stanley97}.  We now explain the general
relation between these simplicial Stirling numbers $S(K,r,s)$ and the
usual Stirling numbers of the second kind.

Define the {\em $s$-monochrome set\/} of a partition $P$ of $V(K)$ to
be the set
\begin{equation*}
  M^s(P) = \{ \sigma \in F^s(K) \mid \text{$\sigma$ is contained in a
    block of $P$} \}
\end{equation*}
of all $s$-simplices entirely contained in one of the blocks of $P$.
The set $M^s(P)$ is an element of the $s$-chromatic lattice $L^s(K)$
by Lemma~\ref{lemma:Mscol}.

\begin{thm}\label{thm:SKrs}
  The number of partitions of $V(K)$ into $r$ $s$-independent blocks
  is
  \begin{equation*}
    S(K,r,s) = \sum_{T \in L^s(K)} \mu(\widehat 0,T) S(|\pi(T)|,r)
  \end{equation*}
 where $\mu$ the \Mb\ function for the $s$-chromatic lattice $L^s(K)$. 
\end{thm}
\begin{proof}
  For any $B \in L^s(K)$, let $S(K,r,s,B)$ be the number of partitions
  $P$ of $V(K)$ into $r$ blocks with monochrome set $M^s(P)=B$. 
  We want to determine $S(K,r,s,\emptyset) = S(K,r,s)$.
  For
  any $A \in L^s(K)$,
  \begin{equation*}
    S(|\pi(A)|,r) = \sum_{A \leq B} S(K,r,s,B)    
  \end{equation*}
  because there are $S(|\pi(A)|,r)$ partitions $P$ of $V(K)$ into $r$
  blocks with $A \leq M^s(P)$. Equivalently,
  \begin{equation*}
      \sum_{A \leq B} \mu(A,B) S(|\pi(B)|,r) = S(K,r,s,A)
  \end{equation*}
  by \Mb\ inversion \cite[Proposition 3.7.1]{stanley97}. The statement
  of the theorem is the particular case of this formula where
  $A=\widehat 0$.  
\end{proof}

\begin{proof}[Proof of Theorem~\ref{lemma:StK}]
  We simply follow the proof of the similar statement for \chp s for
  graphs \cite[Theorem 15]{read68}.  When $r \geq i$ we can get an
  $(r,s)$-coloring out of one of the $S(K,i,s)$ partitions of $V(K)$
  into $i$ $s$-independent blocks by choosing $i$ out of the $r$
  colors and assigning them to the $i$ blocks. There are $\binom ri$
  ways of choosing the $i$ out of $r$ colors and $i!$ ways of coloring
  $i$ blocks in $i$ colors.  The number of $(r,s)$-colorings of $K$ in
  exactly $i$ colors is thus
\begin{equation*}
  S(K,i,s) \binom ri i!  = S(K,i,s) [r]_i 
\end{equation*}
so that
\begin{equation*}
  \chrps Krs = \sum_{i=1}^{m(K)}  S(K,i,s) [r]_i 
\end{equation*}
is the total number of $(r,s)$-colorings of $K$.
\end{proof}

% \begin{lemma}
% The $s$-chromatic polynomial of $K$  is
% \begin{equation*}
%   \chrps Krs = \sum_{i=1}^{m(K)}  S(K,i,s) [r]_i  
% \end{equation*}
% when written in factorial form.
% \end{lemma}

\begin{cor}\label{cor:chiLsK}
  The reduced \euc\ of the open interval $(\widehat 0, \widehat 1)$ in
  $s$-chromatic lattice $L^s(K)$ is
  \begin{equation*}
    \mu(L^s(K))(\widehat 0, \widehat 1) =
    \sum_{i=\chs sK}^{m(K)} (-1)^{i-1}(i-1)!S(K,i,s)
  \end{equation*}
\end{cor}
\begin{proof}
  Equate the terms of degree $1$ of the two expressions 
\begin{equation}\label{eq:twochisKr}
    \sum_{T \in L^s(K)} \mu(\widehat 0, T) r^{|\pi(T)|}
     = \sum_{i=\chs sK}^{m(K)} S(K,i,s)[r]_i
  \end{equation}
  from Theorem~\ref{thm:LsK} and Theorem~\ref{lemma:StK} for the
  $s$-\chp\ of $K$.
\end{proof}

% Alternatively, $S(K,r,s)$ is the number of $(r,s)$-colorings of $K$
% using {\em precisely\/} $r$ indistinguishable colors.

% \begin{cor}\label{cor:StK}
% The sum 
% \begin{equation*}
%    \sum_{i=1}^{r} S(K,i,s)
% \end{equation*}
% counts the number of $(r,s)$-colorings with indistinguishable colors.
% \end{cor}

% \begin{rmk}\label{rmk:FF}
%   The falling factorials
% \begin{equation*}
% [r]_i = r(r-1) \cdots (r-i+1)  
% \end{equation*}
% form a basis for the rational vector space $\Q[r]$ of rational
% polynomials in $r$. In this vector space,
% \begin{equation*}
%   a_0 + a_1r +  a_2r^2 + \cdots = b_0 + b_1[r]_1 + b_2[r]_2 + \cdots
% \end{equation*}
% if and only if
% \begin{equation*}
%   (a_0,a_1,\ldots) = (b_0,b_1,\ldots) (s(i,j))_{0 \leq i,j}, \qquad
%   (b_0,b_1,\ldots) = (a_0,a_1,\ldots) ( S(i,j))_{0 \leq i,j} 
% \end{equation*}
% where $s(i,j)$ and $S(i,j)$ are the Stirling numbers of the first and
% second kind  respectively \cite[p 35]{stanley97}. In particular,
% \begin{equation*}
%   [r]_n =  \sum_{k=0}^n s(n,k)x^k \qquad \text{and} \qquad
%   r^n = \sum_{k=0}^n S(n,k)[r]_k
% \end{equation*}
% \end{rmk}

We observe that
\begin{equation*}
   \sum_i S(K,i,s)[r]_i 
             = \sum_i \sum_T \mu(\widehat 0,T) S(|\pi(T)|,i)[r]_i
             = \sum_T \mu(\widehat 0,T) \sum_i  S(|\pi(T)|,i)[r]_i 
             =  \sum_T \mu(\widehat 0,T) r^{|\pi(T)|}
\end{equation*}
so that Theorem~\ref{thm:SKrs} implies Theorem~\ref{lemma:StK}.

The $s$-chromatic number of $K$ is immediately visible
with the $s$-\chp\ in factorial form because
\begin{equation*}
  \chs sK = \min \{ i \mid S(K,i,s) \neq 0 \}
\end{equation*}
is the lowest degree of the nonzero terms. The positive integer
sequence
\begin{equation*}
  \chrps K{\chs sK}s ,\ldots,\chrps K{m(K)}s = 1
\end{equation*}
has no internal zeros. (If there is a partition of $V(K)$ into $r$
blocks not containing any $s$-simplex of $K$ and $r<m(K)$, then
split one of the blocks with more than one vertex into two sub-blocks
to get a partition of $V(K)$ into $r+1$ blocks containing no
$s$-simplices of $K$.)

The simplicial Stirling numbers satisfy the recurrence relations
%% SKred in chrpol.prg 
\begin{equation*} 
 S(K,r,s) =  \sum_{\substack{\emptyset \subsetneq U \subseteq V(K)-\{
     v_0\} \\ 
     \text{$V(K)-U$ $s$-independent}}} S(K \cap D[U],r-1,s), \qquad
S(K,1,s) =
\begin{cases}
  1 & s> \dim(K) \geq 0  \\ 0 & \text{otherwise} 
\end{cases}
\end{equation*}
To see this, fix a vertex $v_0$ of $K$. 
Let $P$ be  partition of $V(K)$ into $r$ $s$-independent
subsets. Let $U_0$ be the block containing $v_0$. The other blocks in
$P$ form a partition $P_0$ 
of $K \cap D[V(K)-U_0]$ into $r-1$ $s$-independent
subsets. The map $P \leftrightarrow (P_0,U_0)$ is a bijection.

The familiar recurrence relation $S(m,r) = S(m-1,r-1)+rS(m-1,r)$ for
Stirling numbers of the second kind does not readily apply to
simplicial Stirling numbers. The closest analogue may be
\begin{equation*}
  S(K,r,s) = S(K \cap D[V(K)-\{v_0\}],r-1,s) +
  \sum_{P \in  \mathcal{S}(K \cap D[V(K)-\{v_0\}],r,s)}
  |\{ B \in P \mid B \cup \text{$\{v_0\}$ is $s$-independent in $K$} \}|
\end{equation*}
where $v_0$ is a vertex of $K$ and $ \mathcal{S}(K \cap
D[V(K)-\{v_0\},r,s)$ is the set of partitions $P$ of the vertex set of $K
\cap D[V(K)-\{v_0\}$ into $r$ $s$-independent subsets.

\begin{prop}\label{prop:SKisSNis}
Let $K$ be a subcomplex of $L$  and assume that $V(K)=V(L)$.
\begin{enumerate}
\item $S(K,r,s) \geq S(L,r,s)$ for all $r$.\label{prop:SKisSNis1}
\item If $S(K,r,s) = S(L,r,s)$ for some $r$ with
  $\frac{1}{s}(|V|-1) \leq r \leq |V|-s$, then $K^s = L^s$.
  \label{prop:SKisSNis2}
\end{enumerate}
\end{prop}
\begin{proof}
  \noindent \eqref{prop:SKisSNis1} Let $V$ be the vertex set of $K$ and
  $L$. Write $\mathcal{S}(K,r,s)$ and $\mathcal{S}(L,r,s)$ for the set
  of partitions of $V$ into $r$ blocks containing no $s$-simplex of
  $K$ or $L$, respectively. Then $\mathcal{S}(L,r,s) \subseteq
  \mathcal{S}(K,r,s)$ for all $r$ and $s$. Thus $S(L,r,s) \leq
  S(K,r,s)$.
 
  \noindent \eqref{prop:SKisSNis2} Suppose that $\sigma \in
  F^s(L)-F^s(K)$ is an $s$-simplex of $L$ that is not an $s$-simplex
  of $K$. Any partition of the form
  \begin{equation*}
    \{\sigma\} \cup \tau, \qquad \tau \in
    \mathcal{S}(D[V-\sigma],r-1,s), 
  \end{equation*}
  in $\mathcal{S}(K,r,s)-\mathcal{S}(L,r,s)$. The set
  $\mathcal{S}(D[V-\sigma],r-1,s)$ is nonempty when
  \begin{equation*}
    \chs s{D[V-\sigma]} = 
    \left\lceil \frac{|V|-s-1}{s} \right\rceil \leq r-1 \leq |V|-s-1
  \end{equation*}
  and thus $S(K,r,s)$ is strictly greater than $S(L,r,s)$ when
  $\frac{|V|-1}{s} \leq r \leq |V|-s$.
\end{proof}

\begin{rmk}[$S(K,r,s)$ for the complete simplex {$K=D[m]$}]\label{rmk:Smrs}
    %function St in chrpol.prg
  For any finite set $M$, let $S(M,r,s)$ stand for $S(D[M],r,s)$
  (Definition~\ref{defn:SKrs}), the number of partitions of the set
  $M$ into $r$ blocks containing at most $s$ elements. Let us even
  write $S(m,r,s)$ in case $M=[m]$, $m \geq 1$, $r,s \geq 0$.
  %%%%%%%%%%%%%%%%
  % Alternatively, $r!S_m^s(r)$ is the number of surjective maps
  % of an $n$-set onto an $r$-set with fibres that are at most $s$-sets.
  %%%%%%%%%%%%%%%  
  Clearly, $S(m,r,s)$ is nonzero only when $m/s \leq r \leq m$. Also,
  $S(m,r,s) = S(m,r)$ when $r$ is among the $s$ numbers
  $m-s+1,\ldots,m$.
  %%%%%%%%%%%%%%%%%%%
 %  \begin{equation*}
%     S(m,r,s) =
%     \begin{cases}
%       S(m,r) & r > m-s \\ 0 & sr < m \\ 0 & r > m
%     \end{cases}
%   \end{equation*}
%   where $S(m,r)$ is the  Stirling
%   number $S(n,r)$ of the second kind \cite[p $33$]{stanley97}.
  %$n-r+1 \leq s$.
  %(A surjective map $\{1,\ldots,n\}
  %\twoheadrightarrow \{1,\ldots,r\}$ have fibres of size at most
  %$n-r+1$.)  
  %% 
%%%%%%%%%%%%%%%%%%%%%%%
%   In general,
%   \begin{equation*}    
%     S_m^s(r) = \frac{1}{r!} 
%    \sum \binom{m}{m_1,\ldots,m_r} \binom{r}{r_1 \cdots r_s}
%   \end{equation*} 
%   where the sum is taken over all $r$-vectors $(m_1,\ldots,m_r)$ such
%   that $s \geq m_1 \geq \cdots \geq m_r \geq 1$, $m_1+\cdots+m_r=m$,
%   and the $s$-vector $(r_1,\ldots,r_s)$ records the number of fibres
%   of a given size, $r_j = | \{ p \in \{1,\ldots,r\} \mid m_p = j \}|$
%   for $1 \leq j \leq s$.  
The recurrence relation
\begin{equation*}
  S(m,r,s) = \sum_{j=m-s}^{m-1} \binom {m-1}j S(j,r-1,s)
  %% function St in /home/moller/underv/algtop/magma/chrpol.prg
\end{equation*}
can be used to compute these numbers.  Table~\ref{tab:chrtableDm}
shows $S(m,r,s)$ for small $m$; the number $S(m,r,s)$ is in row $s$
and column $r$ in the chromatic table (Definition~\ref{defn:chitable})
for $D[m]$.  All the red numbers are usual Stirling numbers of the
second kind.

%% M-x query-replace-regexp RET [0-9]+ RET \& &

%% M-x query-replace-regexp RET [0-9]+ RET $\&$ &

% \begin{table}[t]
%   \centering
%   \begin{tabular}[t]{cc}
%      \begin{tabular}[c]{c|cccccc}
%      $S(6,r,s)$ & $s=1$ & $s=2$ & $s=3$ & $s=4$ & $s=5$ & $s=6$ 
%    \\ \hline
% $r=1$ &    $0$ &  $0$ &  $0 $ &  $0$ &  $0$ & $1$ \\
%  $r=2$ & $0 $ &  $0 $ & $10 $ & $25 $ & $31 $ & $31$ \\
%  $r=3$ & $0 $ & $15 $ & $75 $ & $90 $ & $90 $ & $90$ \\
%   $r=4$ & $0 $ & $45 $ & $65 $ & $65 $ & $65 $ & $65$ \\
%   $r=5$ & $0 $ & $15 $ & $15 $ & $15 $ & $15 $ & $15$ \\
%  $r=6$ & $1 $ &  $1 $ & $1 $ &  $1 $ &  $1 $ &  $1$
%   \end{tabular} & 
%   \begin{tabular}[c]{c|ccccccc}
%    $S(7,r,s)$ & $s=1$ & $s=2$ & $s=3$ & $s=4$ & $s=5$ & $s=6$ & $s=7$
%    \\ \hline
%    $r=1$ &  $0$ &  $0$  &  $0$ & $0$ & $0$ & $0$ & $1$ \\
%    $r=2$ & $0$  &  $0$ &  $0$ & $35$ & $56$ & $\textcolor{red}{63}$ & $63$ \\
%    $r=3$ & $0$ &  $0$ & $175$ & $280$ & $301$ & $301$ & $301$ \\
%   $r=4$ & $0$ & $105$ & $315$ & $350$ & $350$ & $350$ & $350$ \\
%   $r=5$ & $0$ & $105$ & $140$ & $140$ & $140$ & $140$ & $140$ \\
%   $r=6$ & $0$ & $21$ & $21$ & $21$ & $21$ &  $21$ &  $21$ \\
%    $r=7$ & $1$ &  $1$ &  $1$ &  $1$ &  $1$ &  $1$ &  $1$
%   \end{tabular}
%    \end{tabular}
%   \caption{Examples of $S_m^s(r)$}
%   \label{tab:Snrs}
% \end{table}

\begin{table}[t] %% chitable(D(4))
  \centering
  \begin{tabular}[t]{*{2}{>{$}c<{$}}}
    \begin{pmatrix}
      \textcolor{blue}{0} & \textcolor{red}{1} % \\ 1 & 1 
    \end{pmatrix} &
    \begin{pmatrix}
      0 & \textcolor{blue}{0} & \textcolor{red}{1} \\ 
      \textcolor{blue}{0} & \textcolor{red}{3} 
      & \textcolor{red}{1} % \\ 1 & 3 &3
    \end{pmatrix} \\
    \begin{pmatrix}
      0 & 0 & \textcolor{blue}{0} & \textcolor{red}{1} \\
      0 & \textcolor{blue}{3} & \textcolor{red}{6} & \textcolor{red}{1} \\
      \textcolor{blue}{0} & \textcolor{red}{7} 
      & \textcolor{red}{6} & \textcolor{red}{1} % \\ % 1 & 7 & 6 & 1
    \end{pmatrix} &
    \begin{pmatrix}
       0 & 0 & 0  & \textcolor{blue}{0} & \textcolor{red}{1} \\
       0 & 0 & \textcolor{blue}{15} 
       & \textcolor{red}{10} &  \textcolor{red}{1} \\
       0 & \textcolor{blue}{10} & \textcolor{red}{25} 
      & \textcolor{red}{10} &  \textcolor{red}{1} \\
       \textcolor{blue}{0}  & \textcolor{red}{15} & \textcolor{red}{25} 
   & \textcolor{red}{10} & \textcolor{red}{1} % \\
     %   1 & 15 & 25 & 10 &  1
    \end{pmatrix} \\
    \begin{pmatrix}
     0 & 0 & 0 & 0 & \textcolor{blue}{0} & \textcolor{red}{1} \\
     0 & 0 & 15 & \textcolor{blue}{45} 
     & \textcolor{red}{15} & \textcolor{red}{1} \\
     0  & 10 & \textcolor{blue}{75} & \textcolor{red}{65} 
     & \textcolor{red}{15} &  \textcolor{red}{1} \\
     0 & \textcolor{blue}{25} & \textcolor{red}{90} & \textcolor{red}{65} 
     & \textcolor{red}{15} &  \textcolor{red}{1} \\
     \textcolor{blue}{0} & \textcolor{red}{31} 
     & \textcolor{red}{90} & \textcolor{red}{65} 
     & \textcolor{red}{15} &  \textcolor{red}{1} % \\
    % 1  & 31 & 90 & 65 & 15 &  1  
    \end{pmatrix} &
    \begin{pmatrix}
      0 &  0 &   0 &  0 &  0 &  \textcolor{blue}{0} &  \textcolor{red}{1} \\
      0 &  0 &  0 & 105 & \textcolor{blue}{105} 
      &  \textcolor{red}{21} &   \textcolor{red}{1} \\
      0 &  0 & 175 
      & \textcolor{blue}{315} & \textcolor{red}{140} &  \textcolor{red}{21} 
      &  \textcolor{red}{1} \\
      0 & 35 & \textcolor{blue}{280} & \textcolor{red}{350} 
      & \textcolor{red}{140}
      &  \textcolor{red}{21} &   \textcolor{red}{1} \\
      0 & \textcolor{blue}{56} & \textcolor{red}{301} 
      & \textcolor{red}{350} & \textcolor{red}{140}
      &  \textcolor{red}{21} &  \textcolor{red}{1} \\
      \textcolor{blue}{0} & \textcolor{red}{63} 
      & \textcolor{red}{301} & \textcolor{red}{350}
      & \textcolor{red}{140} &  \textcolor{red}{21} &  \textcolor{red}{1} %\\
     % 1 & 63 & 301 & 350 & 140 & 21 &   1 \\
    \end{pmatrix}
  \end{tabular}
  \caption{Chromatic tables for complete simplices $D[m]$ for $m=2,\ldots,7$ }
  \label{tab:chrtableDm}
\end{table}

According to Theorem~\ref{lemma:StK}, the numbers
$S(m,r,s)$ determine the $s$-\chp\ in falling factorial form
of the complete simplex on $m$ vertices
  \begin{equation*}
    \chrps {D[m]}rs = 
    \sum_{i=\lceil m/s \rceil}^m S(m,i,s) [r]_i
  \end{equation*}
and, according to Corollary~\ref{cor:chiLsK}, they also determine the
reduced \euc\
  %   According to Theorem~\ref{thm:LsK} and Remark~\ref{rmk:wlattice}, 
%     the \chp\ of the complete simplex on $m$ vertices is also given by
%     \begin{equation*}
%       \chrps {D[m]}rs = \sum_{\sigma \in L^s_m(1^m)} 
%       \mu_m^s(1^m)(\widehat 0, \sigma)r^{|\sigma|}
%     \end{equation*}
%     Equating coefficients of the degree $1$ terms of these two
%     polynomials gives the relation
    \begin{equation*}
       \mu_m^s(1^m)(\widehat 0, \widehat 1) =
         \sum_{i=\lceil m/s \rceil}^m (-1)^{i-1}(i-1)! S(m,i,s)   
      %  =\sum_{i=\lceil m/s \rceil}^{m-s}(-1)^{i-1}(i-1)! S_m^s(i) +
      %   \sum_{i=m-s+1}^m (-1)^{i-1}(i-1)! S(m,i)
    \end{equation*}
   of the $s$-chromatic lattice $L^s(D[m])$.

   More generally, if $w \colon M \to \N$ is a function on $M$ with
   natural numbers as values, let $S(M,w,r,s)$ be the number of
   partitions of $M$ into admissible blocks, where we declare a block
   admissible if it is a singleton or it has weight at most $s$. (Then
   $S(m,r,s) = S([m],1^m,r,s)$ occur when $M=[m]$ and $w=1^m$ places
   weight $1$ on all elements.)  Any such partition is a partition of
   $M$ into blocks of weight at most $s$, and therefore $S(M,w,r,s)
   \leq S(\# M,r,s)$.  In particular, $S(M,w,r,s)$ is nonzero only
   when $\# M/s \leq r \leq \# M$. The recurrence relation
    \begin{equation*}
      S(M,w,r,s) = 
      \sum_{\substack{\emptyset \neq J \subset M - \{\max(M)\} \\ 
    \text{$M-J$ admissible}}} S(J,w\vert J,r-1,s)
    \end{equation*}
    provides a means  to compute these numbers.

    The weighted version of Equation~\eqref{eq:twochisKr} for
    $K=D[m]$, 
    \begin{equation*}
       \sum_{\sigma \in L^s_m(w)} 
      \mu_m^s(w)(\widehat 0, \sigma)r^{|\sigma|} =
      \sum_{i= \lceil m/s \rceil}^m S([m],w,i,s) [r]_i
    \end{equation*}
    implies, by equating coefficients of first degree terms, the
    expression
    \begin{equation}\label{eq:umsw01}
      \mu_m^s(w)(\widehat 0, \widehat 1) = 
       \sum_{i= \lceil m/s \rceil}^m (-1)^{i-1}(i-1)!S([m],w,i,s)
    \end{equation}
    for the \euc\ of the weighted lattice $L^s_m(w)$ from
    Remark~\ref{rmk:wlattice}.
  \end{rmk}

  Because any simplicial complex $K$ is a subcomplex of the complete
  simplex $D[m(K)]$ on its vertex set, we have
  \begin{equation}\label{eq:bounds}
    S(m(K),r) \geq S(K,r,s) \geq S(m(K),r,s),  \qquad 1 \leq r \leq m(K)
   %% &  S(N,r)  = S(K,r,s) = S(N,r,s), \qquad N-s+1 \leq r\leq N.
  \end{equation}
  Moreover, these inequalities are equalities for the $s$ highest
  values $m(K)-s+1, \ldots, m(K)$ of $r$. Thus the $s$ terms of
  highest falling factorial degree in the $s$-\chp\ of $K$
\begin{equation*}
  \chrps Krs = \sum_{i=0}^{m(K)-s} S(K,i,s)[r]_i +
                \sum_{i=m(K)-s+1}^{m(K)} S(m(K),i)[r]_i
\end{equation*}
are given by the $s$ Stirling numbers $S(m(K),m(K)-s+1), \ldots,
S(m(K),m(K))$ of the second kind. These coefficients depend only
on the size of the vertex set of $K$. We shall next show that the
coefficient number $s+1$ counted from above, $S(K,m(K)-s,s)$,
informs about the number $f_s(K)$ of $s$-simplices in $K$.

\begin{prop}\label{prop:S(K,m-s,s)}
  $S(K,m(K)-s,s) = S(m(K),m(K)-s)-f_s(K)$. If $S(K,m(K)-s,s) =
  S(m(K),m(K)-s,s)$ then $K^s = D[m(K)]^s$.
\end{prop}
\begin{proof}
  The only partitions of the $S(m,m-s)$ partitions of $V(K)$ into
  $m-s$ blocks that are not $s$-independent are those consisting of
  one $s$-simplex of $K$ together with singleton blocks.  If
  $S(K,m(K)-s,s) = S(D[m(K)],m(K)-s,s)$ then $f_s(K) = f_s(D[m(K)])$
  so $K^s=D[m(K)]^s$. (This is a special case of
  Proposition~\ref{prop:SKisSNis}.\eqref{prop:SKisSNis2}.)
\end{proof}

\begin{defn}\label{defn:chitable}
  The {\em chromatic table}, $\chi(K)$, of $K$ is the $(\dim(K) \times
  m(K))$-table with $S(K,r,s)$ in row $s$ and column $r$.
\end{defn}

This means that row $s$ in the chromatic table lists the coefficients
of the $s$-\chp . The chromatic table of a $3$-dimensional simplicial
complex $K$, for instance, looks like this
\begin{center}
  \begin{tabular}[t]
  {>{$}c<{$}|*{8}{ >{$}c<{$} }} 
  {} & r=1 & r=2 & \cdots  &r=m-3 & r=m-2 & r=m-1 & r=m  \\ \hline
  S(K,\cdot,1) & S(K,1,1)  & S(K,2,1) & \cdots & S(K,m-3,1) 
  & S(K,m-2,1) &
  \textcolor{blue}{S(m,m-1)-f_1} & \textcolor{red}{S(m,m)=1} \\
   S(K,\cdot,2) & S(K,1,2)  & S(K,2,2) & \cdots & S(K,m-3,2) 
   & \textcolor{blue}{S(m,m-2)-f_2} &
  \textcolor{red}{S(m,m-1)} & \textcolor{red}{S(m,m)=1} \\
   S(K,\cdot,3) & S(K,1,3)  & S(K,2,3) & \cdots &
   \textcolor{blue}{S(m,m-3)-f_3}   
   & \textcolor{red}{S(m,m-2)} &
  \textcolor{red}{S(m,m-1)} & \textcolor{red}{S(m,m)=1} \\
  \end{tabular}
\end{center}
where the red entries in row $s$ are Stirling numbers of the second
kind $S(m,r)$ for $m-s+1 \leq r \leq m$, and the blue entry in row $s$
is $S(m(K),m(K)-s)-f_s(K)$.

% \begin{cor}\label{prop:SKisSNis}
%   If $S(K,r,s) = S(N,r,s)$ for all $r \in [N]$ for some $s \geq 1$,
%   then $K^s = D[N]^s$.
% \end{cor}

\begin{exmp}\label{exmp:chitable} %chitable.prg
  The chromatic tables of the $2$-dimensional simplicial complexes
  from Examples~\ref{exmp:K},~\ref{exmp:MB}, and \ref{exmp:MT} are
  \begin{alignat*}{3}
    \chi(K) &=
    \begin{pmatrix}
       0 & 0 & 2 & 10 &  \textcolor{blue}{7} & \textcolor{red}{1} \\  
       0 & 15 & 73 & \textcolor{blue}{62} & \textcolor{red}{15} 
      &  \textcolor{red}{1}
    \end{pmatrix} & &\qquad & 
    \chi(\mathrm{MB}) &=
    \begin{pmatrix}
   0  & 0 &  0 &  \textcolor{blue}{0} & \textcolor{red}{1} \\ 
   0 &  5 & \textcolor{blue}{20} & \textcolor{red}{10} &  \textcolor{red}{1}
    \end{pmatrix} \\
    \chi(\mathrm{MT}) &=
    \begin{pmatrix}
        0 &  0 &  0 &  0 &  0 &  \textcolor{blue}{0} &  
      \textcolor{red}{1}  \\  
        0 &  0 & 84 & 231 & \textcolor{blue}{126} & 
       \textcolor{red}{21} &  \textcolor{red}{1}
    \end{pmatrix} & &\qquad & 
    \chi(\mathrm{P}2) &=
    \begin{pmatrix}
       0 & 0 & 0 &  0 &  \textcolor{blue}{0} & \textcolor{red}{1} \\ 
       0 & 0 & 45 & \textcolor{blue}{55} & \textcolor{red}{15} &  
       \textcolor{red}{1}
    \end{pmatrix}
     \end{alignat*} 
     The red entries in column $r$ are Stirling numbers $S(m,r)$ and
     they are independent of the row index. The blue entry in row $s$
     and column $m-s$, which equals $S(m-s,s)-f_s(K)$, detects if $K$
     has maximal $s$-skeleton by Proposition~\ref{prop:S(K,m-s,s)}.
\end{exmp}

% \begin{exmp} %% colFF3 in chrpol.prg
%   The $2$-\chp\ of the simplicial complexes from
%   Examples~\ref{exmp:K},~\ref{exmp:MB}, and \ref{exmp:MT} are
%   \begin{alignat*}{3}
%     &\chrps Kr2 = 15[r]_2+73[r]_3+62[r]_4+15[r]_5+[r]_6 & &\qquad & 
%     &\chrps {\mathrm{MB}}r2 = 5[r]_2 + 20[r]_3 + 10[r]_4 + [r]_5 \\
%     &\chrps {\mathrm{MT}}r2 = 84[r]_3 +  231[r]_4 +  126[r]_5 + 21[r]_6
%     + [r]_7 & &\qquad & 
%     &\chrps {\mathrm{P}2}r2 = 45[r]_3 + 55[r]_4 + 15[r]_5 + [r]_6 
%   \end{alignat*}
%   when written in factorial form. 
%  %%%%%%%%%%%%%%%%%%%
% %  This table 
% %   \begin{center}
% %   \begin{tabular}{c||c|c}
% %     $r$ & $(r,2)$-colorings &   Indistinguishable $(r,2)$-colorings \\ \hline
% %     $2$ & $10$ & $5$ \\
% %     $3$ & $150$ & $25$ \\
% %     $4$ & $780$ & $35$ \\
% %     $5$ & $2620$ & $36$ \\
% %     $6$ & $6870$ & $36$ 
% %   \end{tabular}
% %     \end{center}
% %     shows the number of $(r,2)$-colorings and the number of
% %     $(r,2)$-colorings with indistinguishable colors for the
% %     triangulated \Mb\ band of Example~\ref{exmp:MB}.
% \end{exmp}

 \begin{exmp}\label{exmp:AS3}
%/home/moller/underv/algtop/magma/tria/AS3
%/home/moller/underv/algtop/magma/chrpol.prg - use the function colFF2
% or colFF3
%% chitable(AS3)
   Let $K=\mathrm{AS}3$ be Altshuler's peculiar triangulation of the 
   $3$-sphere with $f$-vector $f= (10, 45, 70, 35)$ \cite{ALT76}.  The
   $1$-chromatic polynomial is $ \chrps {\mathrm{AS}3}r1 = [r]_{10}$ as
   $K^1$ is the complete graph on $10$ vertices. The chromatic table is
   \begin{equation*}
     \chi(\mathrm{AS}3) =
     \begin{pmatrix}
           0 &    0 &    0 &    0 &    0  &   0  &   0 &    0 &  
       \textcolor{blue}{0}  
           &   \textcolor{red}{1} \\
    0 &    0 &    0 & 1360 & 8475 & 10355 & 4200 &  \textcolor{blue}{680} & 
   \textcolor{red}{45}   
   &    \textcolor{red}{1} \\
    0 &   26 &  4320 &  25915 & 38550 & 22152 & \textcolor{blue}{5845} &   
    \textcolor{red}{750} &  \textcolor{red}{45}  &
    \textcolor{red}{1}
     \end{pmatrix}
   \end{equation*}
   The blue numbers determine the $f$-vector
   \begin{equation*}
     f(\mathrm{AS}3) =(10,
   S(10,9)-\chi(\mathrm{AS}3)_{19},
   S(10,8)-\chi(\mathrm{AS}3)_{28},S(10,7)-\chi(\mathrm{AS}3)_{37})
   \end{equation*}
   The row numbers of the first nonzero term in each row tell us that
   $\chs 1{\mathrm{AS}3} = 10$, $\chs 2{\mathrm{AS}3} = 4$, and $\chs
   3{\mathrm{AS}3} = 2$.
 %  By Equation~\eqref{eq:bounds}, the numbers
%    \begin{align*} %% chrpol.prg function St
%       &(S(10,1,1), \ldots, S(10,10,1)) = 
%      ( 0, 0, 0, 0, 0, 0, 0, 0, 0, 1 ) \\
%      &(S(10,1,2), \ldots, S(10,10,2)) = 
%      ( 0, 0, 0, 0, 945, 4725, 3150, 630, 45, 1 ) \\
%      &(S(10,1,3), \ldots, S(10,10,3)) = 
%      ( 0, 0, 0, 9100, 26145, 19425, 5670, 750, 45, 1 )
%    \end{align*}
%    are lower bounds for the coefficients of the $1$-, $2$, and
%    $3$-\chp s of $K$. 
\end{exmp}

\begin{exmp}\label{exmp:LS3}
  The nonconstructible, nonshellable $3$-sphere $S^3_{17,74}$,
  $f=(17,91,148,74)$, found by Lutz \cite{lutz04}, has 
\begin{gather*}
    \begin{array}{c|*{9}{c}} 
   {} & r=1 & r=2 & r=3 & r=4 & r=5 & r=6 & r=7 & r=8 & r=9 \\ \hline 
   s=1 & 0 & 0 & 0 & 0 & 0 & 0 & 0 & 88 & 3089 \\
   s=2 & 0 & 0 & 36 & 702475 & 82949364 & 1075420155 & 3827766587 &
   5493687086 & 3876597169 \\
   s=3 & 0 & 422 & 4319865 & 338438489 & 3903094622 & 14292381565 
   & 22946854806 & 19158310796 & 9202775199
  \end{array} \\
 \begin{array}{c|*{8}{c}}
  {} & r=10 & r=11 & r=12 & r=13 & r=14 & r=15 & r=16 & r=17 \\ \hline
  s=1 &  23017 & 55285 & 54973 & 25941 & 6210 & 762 & \textcolor{blue}{45} 
  & \textcolor{red}{1} \\
  s=2 &  1507939074 & 346346664 & 48855523 & 4302470 & 235026 
  & \textcolor{blue}{7672} &
  \textcolor{red}{136} & \textcolor{red}{1} \\
  s=3 & 2708454744 & 507528561 & 61784524 & 4903589 & \textcolor{blue}{249826} 
  & \textcolor{red}{7820} &
  \textcolor{red}{136} 
  & \textcolor{red}{1}
 \end{array}
\end{gather*}
as its chromatic table.  Figure~\ref{fig:SKS^3_1774} shows a
semi-logarithmic plot of the simplicial Stirling numbers
$S(S^3_{17,74},r,s)$.
%%%%%%%%%%%%%%%% SLutzS3 S^3_{17,74}
% L1:=[0, 0, 0, 0, 0, 0, 0, 88, 3089, 23017, 55285, 54973, 25941, 6210, 
% 762, 45, 1];
% L2:=[0, 0, 36, 702475, 82949364, 1075420155, 3827766587, 5493687086,
% 3876597169, 1507939074, 346346664, 48855523, 4302470, 235026, 7672,
% 136, 1]; 
% L3:= [0, 422, 4319865, 338438489, 3903094622, 14292381565, 22946854806,
% 19158310796, 9202775199, 2708454744, 507528561, 61784524, 4903589,
% 249826, 7820, 136, 1];
% data for plot generated by para.prg
%%%%%%%%%%%%%%%%%%%%%%%%%%
The triangulation $\Sigma_{16}^3$, $f=(16, 106, 180, 90)$, of the
Poincar\'e homology $3$-sphere constructed by Bj\"orner and Lutz
\cite[Theorem 5]{bjorner-lutz00} has
 \begin{gather*}
  \begin{array}{c|*{8}{c}}
    {} & r=1 & r=2& r=3& r=4& r=5& r=6& r=7& r=8  \\ \hline
    s=1 & 0 & 0 & 0 & 0 & 0 & 0 & 0 & 0 \\
   s=2 & 0 & 0 & 0 & 4589 & 2974411 & 69671411 & 300475213 & 442354547 \\
s=3 & 0 & 3 & 845561 & 70005500 & 701299653 & 2158716508 & 2888730959 &
2000811501 \\ 
   \end{array} \\
 \begin{array}{c|*{8}{c}}
 {} & r=9& r=10& r=11& r=12& r=13& r=14& r=15 & r=16  \\ \hline
 s=1 &  0 & 0 & 0 & 0 & 28 & 44 & \textcolor{blue}{14} & \textcolor{red}{1} \\
 s=2 & 292864435 & 100793551 & 19546606 & 2225261 & 150095 &
 \textcolor{blue}{5840}  &
 \textcolor{red}{120} & \textcolor{red}{1} \\
 s=3 &  792553648 & 190527025 & 28730056 & 2750278 & \textcolor{blue}{165530} &
 \textcolor{red}{6020} & \textcolor{red}{120} & \textcolor{red}{1}
\end{array}
\end{gather*}
as its chromatic table.
%%%%%%%%%%%%%%%%%%% SP3       Poincare sphere 
% L1:=[0, 0, 0, 0, 0, 0, 0, 0, 0, 0, 0, 0, 28, 44, 14, 1];
% L2:=[0, 0, 0, 4589, 2974411, 69671411, 300475213, 442354547,
% 292864435, 100793551, 19546606, 2225261, 150095, 5840, 120, 1];
% L3:=[0, 3, 845561, 70005500, 701299653, 2158716508, 2888730959,
% 2000811501, 792553648, 190527025, 28730056, 2750278, 165530, 6020, 120, 1];
%%%%%%%%%%%%%%%%%%
% and the simplicial $3$-sphere $S^3_{13,56}$ with $f$-vector $f=(13,
% 69, 112, 56)$ and with trivial symmetry group has
% % LS3.log 
% \begin{equation*}
%  % \chi(S^3_{13,56}) = 
%   \left(
%   \begin{array}{*{13}{c}}
%      0 & 0 & 0 & 0 & 0 & 0 & 0 & 0 & 0 & 8 & 17 & 9 & 1 \\
% 0 & 0 & 5 & 8287 & 309311 & 1396057 & 1837553 & 987907 
%  & 254191 & 33657 & 2319 & 78 & 1 \\
% 0 & 29 & 74386 & 1507850 & 5916082 & 8339735 & 5437312 & 1861560 & 357091 & 
% 39269 & 2431 & 78 & 1 
%   \end{array}\right)
% \end{equation*}
% as its chromatic table.
%%%%%%%%%% LS3b.log  \Sigma^3_{16,90}
% L1:=[0, 0, 0, 0, 0, 0, 0, 0, 0, 8, 17, 9, 1]
% L2:=[0, 0, 5, 8287, 309311, 1396057, 1837553, 987907, 254191, 33657, 2319,
% 78, 1]
% L3:=[0, 29, 74386, 1507850, 5916082, 8339735, 5437312, 1861560, 357091, 
% 39269, 2431, 78, 1]
%%%%%%%%%%%%%%%%%%%
\end{exmp}

\begin{figure}[t]
  \centering
  \begin{tikzpicture}  %% S^3_{17,74}
[vertex/.style= {shape=circle,  fill=#1!100, minimum size =
  6pt, inner sep =0pt,draw=#1},
vertex/.default=black,scale=.75]

\draw[very thin, color=gray,opacity=.3] (-0.1,-0.1) grid (18.5,11.5);
\draw[->] (-0.2,0) -- (18.5,0) node[right] {$r$};
\draw[->] (0,-0.2) -- (0,11.5) node[above] {$\log_{10} S(S^3_{17,74},r,s)$};

\foreach \x in {1,2,3,4,5,6,7,8,9,10,11,12,13,14,15,16,17}
\draw (\x,-0.5) node {$\x$};

\foreach \y in {1,2,3,4,5,6,7,8,9,10,11}
\draw (-0.5,\y) node {$\y$};

\draw  plot[smooth] coordinates
{
(8, 1.9)
(9, 3.5)
(10, 4.4)
(11, 4.8)
(12, 4.8)
(13, 4.4)
(14, 3.8)
(15, 2.9)
(16, 1.7)
(17, 0.00)};
\draw (8,1.9) node[below] {$s=1$};
\foreach \x in 
{(8, 1.9),
(9, 3.5),
(10, 4.4),
(11, 4.8),
(12, 4.8),
(13, 4.4),
(14, 3.8),
(15, 2.9),
(16, 1.7),
(17, 0.00)}
\draw \x node[vertex=red] {};

\draw  plot[smooth] coordinates
{(3, 1.6)
(4, 5.9)
(5, 7.9)
(6, 9.0)
(7, 9.6)
(8, 9.8)
(9, 9.6)
(10, 9.1)
(11, 8.5)
(12, 7.7)
(13, 6.6)
(14, 5.4)
(15, 3.9)
(16, 2.1)
(17, 0.00)};
\foreach \x in 
{(3, 1.6),
(4, 5.9),
(5, 7.9),
(6, 9.0),
(7, 9.6),
(8, 9.8),
(9, 9.6),
(10, 9.1),
(11, 8.5),
(12, 7.7),
(13, 6.6),
(14, 5.4),
(15, 3.9),
(16, 2.1),
(17, 0.00)}
\draw \x node[vertex=red] {};
\draw (3,1.6) node[below] {$s=2$};

\draw  plot[smooth] coordinates
{(2,2.62)
(3,6.63)
(4,8.53)
(5,9.59)
(6,10.2)
(7,10.4)
(8,10.3)
(9,9.97)
(10,9.44)
(11,8.70)
(12,7.79)
(13,6.69)
(14,5.40)
(15,3.89)
(16,2.13)
(17,0.000)};
\foreach \x in 
{(2,2.62),
(3,6.63),
(4,8.53),
(5,9.59),
(6,10.2),
(7,10.4),
(8,10.3),
(9,9.97),
(10,9.44),
(11,8.70),
(12,7.79),
(13,6.69),
(14,5.40),
(15,3.89),
(16,2.13),
(17,0.000)}
\draw \x node[vertex=red] {};
\draw (2,2.62) node[below] {$s=3$};
\end{tikzpicture}
  \caption{The simplicial Stirling numbers for $S^3_{17,74}$}
  \label{fig:SKS^3_1774}
\end{figure}
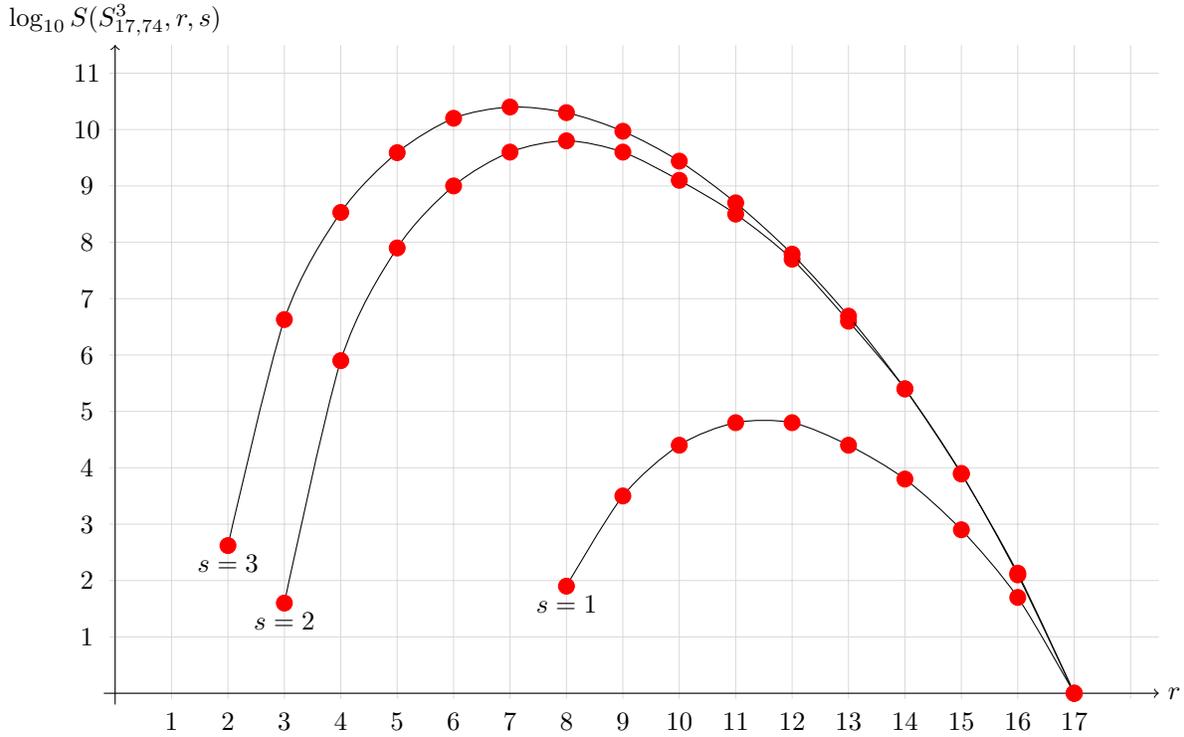

Observe that all the above chromatic tables have strictly
$\log$-concave rows.

 \begin{defn}\cite{sagan88}\label{defn:logconcave}
   A finite sequence $a_1,a_2,\ldots,a_N$ of $N \geq 3$ nonnegative
   integers is strictly $\log$-concave if $a_{i-1}a_{i+1} < a_i^2$ for
   $1<i<N$ (and $\log$-concave if $a_{i-1}a_{i+1} \leq a_i^2$).
 \end{defn}

 It has been conjectured that the sequence of coefficients of the
 $1$-\chp\ of a simple graph in falling factorial form, $r \to
 S(K,1,r)$, $ \chs 1K \leq r \leq m(K)$, is $\log$-concave
 \cite[Conjecture 3.11]{brenti94}. More generally, one may ask

\begin{question}\label{question}
  Is the finite sequence of simplicial Stirling numbers 
 \begin{equation*}
   r \to S(K,r,s), \qquad \chs sK \leq r \leq m(K),
 \end{equation*} 
  $\log$-concave for fixed $K$ and $s$?
\end{question}

This seems to be the right question to ask as it may be true for {\em
  all\/} the \chp s of a simplicial complex and we have seen that the
absolute value of the coefficients of the $s$-\chp\ are simply not
$\log$-concave for $s>1$.

Note that the Stirling numbers of the second kind, which are upper
bounds for the simplicial Stirling numbers $S(K,r,s)$ by the
inequalities~\eqref{eq:bounds}, are $\log$-concave in $r$
\cite[Corollary 2]{sagan88}.

We shall now examine Question~\ref{question} on two spherical boundary
complexes of cyclic $n$-polytopes.

\begin{defn}
  $\partial\mathrm{CP}(m,n)$, $m>n$, is the $(n-1)$-dimensional simplicial
  complex on the ordered set $[m]$ with the following
  facets: An $n$-subset $\sigma$ of $[m]$ is a facet if and only if
  between any two elements of $[m]-\sigma$ there is an even number of
  vertices in $\sigma$.
\end{defn}

By Gale's Evenness Theorem \cite{gale63}, the simplicial complex
$\partial\mathrm{CP}(m,n)$ triangulates the boundary of the cyclic
$n$-polytope on $m$ vertices.  Thus $\partial\mathrm{CP}(m,n)$ is a simplicial
$(n-1)$-sphere on $m$ vertices and it is $\lfloor n/2
\rfloor$-neighborly in the sense that $\partial\mathrm{CP}(m,n)$ has the same
$s$-skeleton as the full simplex on its vertex set when $s < \lfloor
n/2 \rfloor$. 
% When $s < \lfloor n/2
% \rfloor$, the $s$-chromatic number of the cyclic polytope
% $\partial\mathrm{CP}(m,n)$, $\chs s{\partial\mathrm{CP}(m,n)} = \chs s{D[V]} = \lceil
% m/s \rceil$, grows with the number of vertices $m$. The first
% interesting chromatic numbers for the cyclic polytopes are $\chs
% {n}{\partial\mathrm{CP}(m,2n)}$ and $\chs {n}{\partial\mathrm{CP}(m,2n+1)}$.

\begin{exmp}[Cyclic polytopes with $\log$-concave simplicial Stirling numbers
  of the second kind]
  Let $\partial\mathrm{CP}(m,n)$ be the triangulated boundary of the
  cyclic polytope on $m$ vertices in $\R^n$. The simplicial complex
  $\partial\mathrm{CP}(m,n)$ is an $m$-vertex triangulation of
  $S^{n-1}$. The chromatic tables of the simplicial $3$-spheres
  $\partial\mathrm{CP}(m,4)$ on $m=6,7,8,9,10$ vertices are
  \begin{gather*}
    \begin{pmatrix}
      0 & 0 & 0 & 0 
      & \textcolor{blue}{0} & \textcolor{red}{1} \\
      0 & 1 & 21 
      & \textcolor{blue}{47} & \textcolor{red}{15} & \textcolor{red}{1} \\
      0 & 16 
      & \textcolor{blue}{81} & \textcolor{red}{65} 
      & \textcolor{red}{15} &  \textcolor{red}{1}
    \end{pmatrix}
    \begin{pmatrix}
        0 &  0 &  0 &  0 &  0 
        &  \textcolor{blue}{0} &  \textcolor{red}{1} \\
        0 &  0 & 28 & 147 
        & \textcolor{blue}{112} 
        & \textcolor{red}{21} &   \textcolor{red}{1} \\
        0 & 21 & 238 & \textcolor{blue}{336} & \textcolor{red}{140} 
        & \textcolor{red}{21} &  \textcolor{red}{1}
    \end{pmatrix}
    \begin{pmatrix}
         0 &   0 &   0 &   0 &   0  &  0 
         & \textcolor{blue}{0}   &   \textcolor{red}{1} \\
   0  &  1 &   50 & 393 &  582 
   & \textcolor{blue}{226}  
   &  \textcolor{red}{28}  &   \textcolor{red}{1} \\
   0  & 29 & 654 & 1533 
   & \textcolor{blue}{1030} &  \textcolor{red}{266} 
   &  \textcolor{red}{28} &  \textcolor{red}{1} \\
    \end{pmatrix} \\
    \begin{pmatrix}
         0 &   0 &    0 &   0 &   0 &   0 &   0 
         &   \textcolor{blue}{0} &  \textcolor{red}{1} \\
   0 &   0 &  94 & 1062 & 2523 & 1719 
   &  \textcolor{blue}{408} 
    & \textcolor{red}{36}  &   \textcolor{red}{1} \\
   0 &  36 & 1729 & 6471 & 6591 
   & \textcolor{blue}{2619} &  \textcolor{red}{462} 
    & \textcolor{red}{36} &   \textcolor{red}{1}
    \end{pmatrix}
    \begin{pmatrix}
         0 &    0 &    0  &   0 &    0 &    0 &    0 &    0 
         & \textcolor{blue}{0} &
         \textcolor{red}{1} \\
    0 &    1 &  180 &  2980 & 10200 & 10777 &  4225 
    &  \textcolor{blue}{680} 
    & \textcolor{red}{45}  &
    \textcolor{red}{1} \\
   0 &   46 &  4445 & 25960 & 38550 & 22152 
   &  \textcolor{blue}{5845}
   &  \textcolor{red}{750} &   \textcolor{red}{45} & \textcolor{red}{1}
    \end{pmatrix}
  \end{gather*}
  All rows are strictly $\log$-concave. As
  $\partial\mathrm{CP}(m,4)^1=D[m]^1$, the $1$-chromatic number $\chs
  1 {\partial\mathrm{CP}(m,4)} = m$, and it is not difficult to see
  that the $2$-chromatic number $\chs 2 {\partial\mathrm{CP}(m,4)}$ is
  $2$ if $m$ is even and $3$ if $m$ is odd \cite{DMN2010}.
%%%%%%%%%%%%%%%%%%%%%%%%%%%
% The coefficients of the $s$-\chp s of
%   $\partial\mathrm{CP}(6,4)$ and $\partial\mathrm{CP}(7,4)$, the columns in the tables
%   %% function St(K,r,s) in chrpol.prg
%   %% colFF3(CP(6,5),3); 
%   \begin{center}
%   \begin{tabular}[t]{cc}
%   \begin{tabular}[t]{c|ccc}
%    $S(CP(6,4),r,s)$ & $s=1$ & $s=2$ & $s=3$ \\ \hline
%    $r=1$ & $0$ & $0$ & $0$ \\
%    $r=2$ & $0$ & $1$ & $16$ \\
%    $r=3$ & $0$ & $21$ & $81$ \\
%    $r=4$ & $0$ & $47$ & $\textcolor{red}{65}$ \\
%    $r=5$ & $0$ & $\textcolor{red}{15}$ & $\textcolor{red}{15}$ \\
%    $r=6$ & $\textcolor{red}{1}$ & $\textcolor{red}{1}$ & $\textcolor{red}{1}$
%    \end{tabular} &
%    \begin{tabular}[t]{c|ccc}
%    $S(CP(7,4),r,s)$ & $s=1$ & $s=2$ & $s=3$ \\ \hline
%    $r=1$ & $0$ & $0$ & $0$ \\
%    $r=2$ & $0$ & $0$ & $21$ \\
%    $r=3$ & $0$ & $28$ & $238$ \\
%    $r=4$ & $0$ & $147$ & $336$ \\
%    $r=5$ & $0$ & $112$ & $\textcolor{red}{140}$ \\
%    $r=6$ & $0$ & $\textcolor{red}{21}$ & $\textcolor{red}{21}$ \\
%    $r=7$ & $\textcolor{red}{1}$ & $\textcolor{red}{1}$ & 
%    $\textcolor{red}{1}$
%    \end{tabular}
%   \end{tabular}
%   \end{center}
%   are strictly $\log$-concave sequences.
%% colFF3(CP(6,5),3);
%%%%%%%%%%%%%%%%% 
  \end{exmp}

Right multiplication with the upper triangular matrix $([j]_i)_{1 \leq
  i,j \leq m(K)}$ with $[j]_i = \binom ji i! = \frac{j!}{(i-j)!}$ in
row $i$ and column $j$ transforms, by Theorem~\ref{lemma:StK}, the
chromatic table into the $(\dim(K) \times m(K))$-matrix
\begin{equation*}
  \chi(K) ([j]_i)_{1 \leq i,j \leq m(K)} =
  \left( \chi^s(K,i) \right)_{\substack{1 \leq s \leq \dim(K) \\ 1
      \leq i \leq m(K)}}
\end{equation*}
with the $m(K)$ values $\chi^s(K,i)$, $1 \leq i \leq m(K)$, of the
$s$-\chp\ in row $s$. This matrix of chromatic polynomial values
appears also to have $\log$-concave rows.
%% chitable(MB)*Transpose(FFmatrix(5)); in chrpol.prg
% There seems to be a lot of structure in
% Factorial(m)*FFmatrix(m)^-1 - in the rows we have n!/k!

\section{Chromatic uniqueness}
\label{sec:chrunique}

In this section we briefly discuss to what extent simplicial complexes
are determined by their \chp s.  Proposition~\label{prop:S(K,m-s,s)}
shows that the chromatic table of a simplicial complex determines its
$f$-vector.

\begin{defn}\label{defn:unique}
  $K$ is {\em chromatically unique\/} if it is determined up to iso\m\
  by its chromatic table.
\end{defn}

% In other words, $K$ is chromatically unique if determined by its
% $s$-\chp s $\chrps Kr1$, $\chrps Kr2$, \ldots, $\chrps Krd$ where $d =
% \dim(K)$ is the dimension of $K$.

In Lemma~\ref{lemma:KveeL} below, $K \amalg L$ is
the disjoint union and $K \vee L$ the one-point union of $K$ and $L$.
The proof is identical to the one for the similar statements about
chromatic polynomials for simple graphs.

\begin{lemma}\label{lemma:KveeL}
  If $K$ and $L$ are finite simplicial complexes then
  \begin{equation*}
    \chrps {K \amalg L}rs = \chrps Krs \chrps Lrs, \qquad
    \chrps {K \vee L}rs = \frac{\chrps Krs \chrps Lrs}{r}
  \end{equation*}
  for all $r$ and all $s \geq 0$.
\end{lemma}

The two nonisomorphic simplicial complexes
% We can use Lemma~\ref{lemma:KveeL} to produce examples of
% chromatically equivalent simplicial complexes. We just need to find
% examples of finite simplicial complexes $K$ and $L$ so that the iso\m\
% type of $K \vee L$ depends on the choice of point where $K$ and $L$
% are joined.
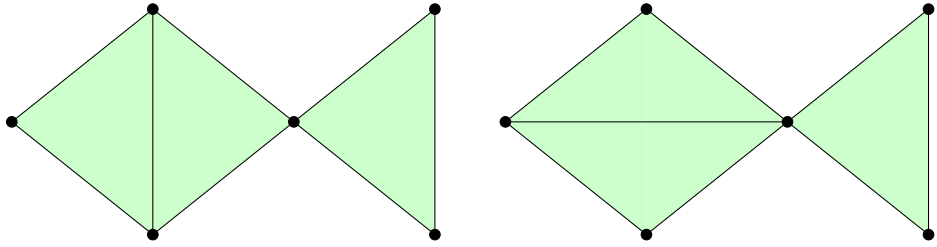
\begin{center}
 \begin{tikzpicture}
[vertex/.style= {shape=circle,  fill=#1!100, minimum size =
  4pt, inner sep =0pt,draw}, vertex/.default=black,scale=0.75,xscale=1.25]

\begin{scope}
    \path [fill = green!20] (0,0) -- (2,2) -- (2,-2) -- (0,0); 
  \path [fill = green!20] (4,0) -- (2,2) -- (2,-2) -- (4,0); 
  \path [fill = green!20] (4,0) -- (6,2) -- (6,-2) -- (4,0);

  \draw (2,2) -- (2,-2);
  \draw (0,0) -- (2,2) -- (4,0) -- (6,2) -- (6,-2) -- (4,0) -- (2,-2)
  -- (0,0) -- cycle;

  \draw (0,0) node[vertex] {};
  \draw (2,2) node[vertex] {};
  \draw (2,-2) node[vertex] {};
  \draw (4,0) node[vertex] {};
  \draw (6,-2) node[vertex] {};
  \draw (6,2) node[vertex] {};
\end{scope}

\begin{scope}[xshift=7cm]
    \path [fill = green!20] (0,0) -- (2,2) -- (2,-2) -- (0,0); 
  \path [fill = green!20] (4,0) -- (2,2) -- (2,-2) -- (4,0); 
  \path [fill = green!20] (4,0) -- (6,2) -- (6,-2) -- (4,0);

  \draw (0,0) -- (4,0);
  \draw (0,0) -- (2,2) -- (4,0) -- (6,2) -- (6,-2) -- (4,0) -- (2,-2)
  -- (0,0) -- cycle;

  \draw (0,0) node[vertex] {};
  \draw (2,2) node[vertex] {};
  \draw (2,-2) node[vertex] {};
  \draw (4,0) node[vertex] {};
  \draw (6,-2) node[vertex] {};
  \draw (6,2) node[vertex] {};
\end{scope}
\end{tikzpicture}
\end{center}
are not chromatically unique as they have identical chromatic tables
\begin{equation*}
  \begin{pmatrix}
     0 &  0 & 2 &10 & \textcolor{blue}{7} & \textcolor{red}{1} \\
     0  & 15 & 73 
     & \textcolor{blue}{62} & \textcolor{red}{15} &  \textcolor{red}{1}
  \end{pmatrix}
\end{equation*}
by Lemma~\ref{lemma:KveeL}. (These two complexes are, however,
PL-isomorphic.)

On the other hand,
Proposition~\ref{prop:SKisSNis}.\eqref{prop:SKisSNis2} immediately
implies that the $s$-skeleton of a full simplex is chromatically
unique (in a very strong sense).

\begin{prop}
  If $K$ has the same $s$-\chp\ as a full simplex $D[N]$, then $K$ and
  $D[N]$ have isomorphic $s$-skeleta.
\end{prop}
\begin{proof}
  If $K$ and $D[N]$ have the same $s$-\chp\ for some $s \geq 1$, then
  $K$ has $N$ vertices (Corollary~\ref{cor:LsK}), and, since
  $\chrps K{N-s}s = \chrps {D[N]}{N-s}s$, the $s$-skeleton of $K$ is
  isomorphic to the $s$-skeleton of the full simplex on $N$ vertices
  (Proposition~\ref{prop:SKisSNis}.\eqref{prop:SKisSNis2}).
\end{proof}

\def\cprime{$'$}
\providecommand{\bysame}{\leavevmode\hbox to3em{\hrulefill}\thinspace}
\providecommand{\MR}{\relax\ifhmode\unskip\space\fi MR }
\providecommand{\MRhref}[2]{%
  \href{http://www.ams.org/mathscinet-getitem?mr=#1}{#2}
}
\providecommand{\href}[2]{#2}

% \bibliographystyle{amsplain}
% \bibliography{/home/moller/projects/top} %% earlier simpl.bib

 \end{document}